\documentclass[12pt,reqno]{amsart}
\usepackage{latexsym,amsthm,amscd, amsmath,amsfonts,amssymb,euscript,hyperref,graphics,color}
\usepackage{graphicx}
\usepackage{comment}
\usepackage{import}
\usepackage{tikz}
\usepackage{latexsym}

\def\ub{\underline{u}}
\def\Lb{\underline{L}}
\def\Cb{\underline{C}} \def\Eb{\underline{E}}
\def\Fb{\underline{F}}
\def\gslash{\mbox{$g \mkern -8mu /$ \!}}
\def\doubleint{\int\!\!\!\!\!\int}
\def\nablaslash{\mbox{$\nabla \mkern -13mu /$ \!}}
\def\laplacianslash{\mbox{$\triangle \mkern -13mu /$ \!}}

\newtheorem*{Main A priori Estimates}{Main A priori Estimates}
\newtheorem*{Main Theorem}{Main Theorem}
\newtheorem{theorem}{Theorem}[section]
\newtheorem{lemma}[theorem]{Lemma}
\newtheorem{proposition}[theorem]{Proposition}

\newtheorem{remark}[theorem]{Remark}

\numberwithin{equation}{section}

\begin{document}
\title[2+1 Wave Map]{Long Time Solutions for Wave Maps with Large Data}

\author[Jinhua Wang]{Jinhua Wang}
\address{Center of Mathematical Sciences, Zhejiang University\\ Hangzhou, China}
\email{youcky0208@gmail.com}

\author[Pin Yu]{Pin Yu}
\address{Mathematical Sciences Center, Tsinghua University\\ Beijing, China}
\email{pin@math.tsinghua.edu.cn}
\thanks{JW is deeply indebted to Professors Dexing Kong and Kefeng Liu for the encouragement and guidance. She would like to thank the Mathematical Sciences Center of Tsinghua
University where the work was partially done during her visit.\\
\indent PY is supported by NSF-China Grant 11101235. He would like to thank Prof. Sergiu Klainerman for the communication of the ideas on the relaxation of the propagation estimates.}

\begin{abstract}
For $2+1$ dimensional wave maps with $\mathbb{S}^2$ as the target, we show that for all positive numbers $T_0 > 0$ and $E_0 >0$, there exist Cauchy initial data with energy at least $E_0$, so that the solution's life-span is at least $[0,T_0]$. We assume neither symmetry nor closeness to harmonic maps.
\end{abstract}
\maketitle
%\tableofcontents

\section{Introduction}\label{introduction}
The $2+1$ dimensional wave map problem with $\mathbb{S}^2$ as the target is the study of maps
\begin{equation*}
\varphi: \mathbb{R}^{2+1} \rightarrow \mathbb{S}^2,
\end{equation*}
satisfying the following system of semi-linear wave equations
 \begin{equation}\label{Wave Map Equation}
 \Box \varphi = ( -|\partial_t \varphi|^2 + \sum_{i=1}^2|\partial_{x_i} \varphi|^2)\varphi.
 \end{equation}
Since we regard $\mathbb{S}^2$ as the unit sphere embedded in $\mathbb{R}^3$, the map $\varphi$ has three components $\varphi_1(t,x), \varphi_2(t,x)$ and $\varphi_3(t,x)$. Therefore, the system \eqref{Wave Map Equation} can be explicitly written as
\begin{equation*}
\Box \varphi_k = ( -\sum_{j=1}^3 |\partial_t \varphi_j|^2 + \sum_{j=1}^3\sum_{i=1}^2|\partial_{x_i} \varphi_j|^2)\varphi_k,
\end{equation*}
for $k= 1,2,3$. For the sake of simplicity, we rewrite \eqref{Wave Map Equation} as
 \begin{equation}\label{Main Equation}
 \Box \varphi =\varphi \cdot Q_0(\nabla \varphi, \nabla \varphi),
 \end{equation}
where we use $Q_0$ to denote a specific null form (see Section \ref{section null form} for definitions). The initial data are given by $(\varphi,\partial_t \varphi)|_{t=0} = (\varphi^{(0)}, \varphi^{(1)})$, where $\varphi^{(0)}:\mathbb{R}^2 \rightarrow \mathbb{S}^2$ and $\varphi^{(1)}:\mathbb{R}^2 \rightarrow T_{\varphi^{(0)}}\mathbb{S}^2$.

We remark that the system \eqref{Main Equation} (with more general targets) is invariant with respect to the scaling $\varphi (t, x) \rightarrow \varphi (\lambda t,  \lambda x)$ for $\lambda \in \mathbb{R}_{>0}$. The energy (or $\dot{H}^1$ norm) of the solution is a dimensionless quantity with respect to this scaling; we refer to such a wave map problem on $\mathbb{R}^{2+1}$ as an \emph{energy critical} problem.

We now state the main result of the paper:
\begin{Main Theorem}
For any given positive numbers $T_0 > 0$ and $E_0 >0$, there exists a smooth initial data set $(\varphi^{(0)}, \varphi^{(1)})$ for the wave map system \eqref{Wave Map Equation}, such that the energy
\begin{equation*}
\text{Energy}_{(1)}(\varphi^{(0)}, \varphi^{(1)}) = \frac{1}{2} \int_{\mathbb{R}^2} |\nabla\varphi^{(0)}|^2 + |\varphi^{(1)}|^2 d x_1 d x_2 \geq E_0,
\end{equation*}
and the life-span of the solution is at least $[0,T_0]$.
\end{Main Theorem}
\begin{remark}
Moreover, if for $k\in \mathbb{N}$ we define the $k$-th order energy of the data as
\begin{equation*}
\text{Energy}_{(k)}(\varphi^{(0)}, \varphi^{(1)}) = \frac{1}{2} \int_{\mathbb{R}^2} |\nabla^{k}\varphi^{(0)}|^2 + |\nabla^{k-1}\varphi^{(1)}|^2 d x_1 d x_2,
\end{equation*}
we can show that
\begin{equation*}
\text{Energy}_{(k)}(\varphi^{(0)}, \varphi^{(1)}) \geq \delta^{-(k-1)},
\end{equation*}
where $\delta$ is a small positive parameter. We note in passing that the higher order energies can be extremely large.
\end{remark}

We briefly summarize the progress on the local well-posedness (LWP) and global well-posedness (GWP) for wave maps (on $\mathbb{R}^{n+1}$ with general Riemannian manifolds as targets). In the subcritical case where data are in $H^s(s>n)$, pioneering works on the LWP were due to Klainerman-Machedon (\cite{K-M-95}, \cite{K-M-97-1}, \cite{K-M-97-2}) and Klainerman-Selberg (\cite{K-S-97},\cite{K-S-02}). For the critical case, one of the first results was obtained by Tataru (\cite{T-98}, \cite{T-01}) in the critical Besov spaces $\dot{B}^{\frac{n}{2}}_{2,1} \times \dot{B}^{\frac{n}{2}-1}_{2,1}$; the LWP and small data GWP in energy spaces with $\mathbb{S}^2$ as the target have been established by Tao \cite{TT-1},\cite{TT-2}. Further contributions for other target manifolds came from the works by Klainerman-Rodnianski \cite{K-R-wm}, Nahmod-Stephanov-Uhlenbeck \cite{N-S-U}, Tataru (\cite{T-04}, \cite{T-05}), and Krieger (\cite{K-03-1}, \cite{K-03-2}, \cite{K-04}). For large data, GWP has been established independently by three groups: Sterbenz and Tataru (\cite{S-T-1},\cite{S-T-2}) obtained GWP for arbitrary target manifolds with initial data below the energy of any nontrivial harmonic map; Tao (\cite{TT-3}-\cite{TT-7}) obtained GWP for the target manifold $\mathbb{H}^2$ and Krieger-Schlag \cite{K-S} established the same result for a negatively curved Riemann surface.

The blow-up phenomena for the system \eqref{Main Equation} have also been extensively studied. For $n \geq 3$, Shatah \cite{Shatah} showed the existence of self-similar blow-up solutions (when $n=2$, there is no such solution). Based on earlier works of Christodoulou and Tavildar-Zadeh (\cite{C-T-1}, \cite{C-T-2}), Shatah and Tahvildar-Zadeh (\cite{Shatah-T-92}, \cite{Shatah-T-94}); Struwe \cite{Struwe} showed that blow up in the $SO(2,\mathbb{R})$ equivariant case must result from rescaling of a harmonic sphere. Recent works of  Rodnianski-Sterbenz \cite{R-S} and Krieger-Tataru \cite{K-S-T} exhibited finite time blow-up solutions. Further investigation by Rodnianski-Raphael \cite{R-R} yielded sharp asymptotic on the dynamics at blow-up time; moreover, they proved the quantization of the energy concentrated at the singularity. We shall remark in passing that the initial data for the aforementioned blow-up results are from special perturbations of the harmonic maps.\\

Inspired by two recent works on the dynamical formation of black holes by Christodoulou \cite{Ch-08} and Klainerman-Rodnianski \cite{K-R-09}, we propose an alternative approach to study the long time existence of wave maps for large data. We will briefly recall their works.

In \cite{Ch-08}, Christodoulou discovered a remarkable mechanism responsible for the dynamical formation of black holes. He showed that a trapped surface can form, even in vacuum space-time, from completely dispersed initial configurations and by means of the focusing effect of gravitational waves. He identified an open set of initial data (so called \emph{short pulse ansatz}) without any symmetry assumptions. Although the data are no longer close to Minkowski data, he could still prove a long time existence result for these data. This establishes the first result on the long time dynamics in general relativity and paves the way for many new developments on dynamical problems relating to black holes.

In \cite{K-R-09}, Klainerman and Rodnianski extended and significantly simplified Christodoulou's work. A key ingredient in their paper is the relaxed propagation estimates, namely, if one enlarges the admissible set of initial conditions, the corresponding propagation estimates are much easier to derive. They reduced the number of derivatives needed in the estimates from two derivatives on curvature (in Christodoulou's proof) to just one. We should note that the direct consequence of the simpler proof of Klainerman-Rodnianski yields results weaker than those obtained by Christodoulou. In fact, within those more general initial data set, they can only show long time existence results for vacuum Einstein field equations; nevertheless, once such existence results are obtained, one can improve them by assuming more on the data, say, consistent with Christodoulou's assumptions and then one can derive Christodoulou's results in a straightforward manner.\\

In our current work, the choice of initial data will be analogous to the short pulse ansatz in \cite{Ch-08}; the proof will rely on a relaxed version of energy estimates similar to the relaxation of the propagation estimates in \cite{K-R-09}. These ideas can be generalized to $3+1$ dimensions. In fact, one can show that for semi-linear wave equations with general null forms nonlinearities on $\mathbb{R}^{3+1}$, there is an open set of large data which allow global solutions in the future. This will be the subject of a forthcoming paper \cite{WY} by the authors.

\section{Preliminaries}\label{Section Preliminaries}

\subsection{Geometric Preliminaries}
We review some geometric constructions on Minkowski space $\mathbb{R}^{2+1}$. Besides the standard Cartesian coordinates $(t, x_1,x_2)$, we will mainly use the null-polar coordinates $(u,\ub,\theta)$. Let $r = \sqrt{x_1^2 + x_2^2}$ be the spatial radius function, two optical functions $u$ and $\ub$ are defined as
\begin{equation*}
u = \frac{1}{2}(t-r), \quad \ub = \frac{1}{2}(t+r).
\end{equation*}
The angular coordinate $\theta$ denotes a point on unit circle $\mathbb{S}^1 \subset \mathbb{R}^2$.  We also use $C_{c}$ to denote the the level surface of the function $u = c $; similarly, $\Cb_{\ub}$ denotes a level set of $\ub$. Their geometric pictures are cones and their intersection $C_u\cap \Cb_{\ub}$ will be a circle denoted by $S_{\ub,u}$.

We use $L$ and $\Lb$ to denote the following future-pointed null vector fields:
\begin{equation*}
 L = \partial_t + \partial_r, \quad \text{and} \quad \Lb = \partial_t - \partial_r.
\end{equation*}
We use $\Omega$ to denote the rotation vector field:
\begin{equation*}
\Omega = x_1 \partial_2 - x_2 \partial_1.
\end{equation*}
In fact, $\Omega = \partial_\theta$ is a Killing vector field. We also use $\nablaslash$ to denote the derivative on $S_{\ub,u}$ with respect to the arc-length on $S_{\ub,u}$. Therefore, $\Omega = r \nablaslash$. In particular, for a given $k \in \mathbb{Z}_{\geq 0}$ and a smooth function $\phi$, we have
\begin{equation}\label{Compare Omega and nablaslash}
 |\Omega^k \phi| = |r|^k |\nablaslash^k \phi|.
\end{equation}

In Section \ref{Section Initial Data in Short Pulse Regime} and Section \ref{Section A priori Estimates}, which are the technical heart of the paper, the parameter $u$ will be confined in the interval $[u_0, -1]$ where $u_0 \sim -T_0$. The parameter $\ub$ is confined in $[u_0, \delta]$ where $\delta$ is small positive parameter which will be determined later. The corresponding cones are pictured as follows (on the left):

\includegraphics[width=5.5in]{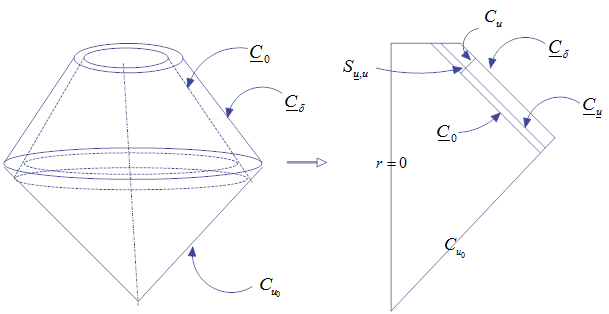}

To simplify, we can forget the $\theta$ direction and draw a simplified two dimensional picture as on the right of the above picture. When we derive estimates in Section \ref{Section Initial Data in Short Pulse Regime} and \ref{Section A priori Estimates}, $\ub \in [0, \delta]$ where $\delta$ will be sufficiently small. Since $T_0$ and $u_0$ are fixed numbers, in the region where $(\ub,u) \in [0,\delta] \times [u_0, -1]$, the parameter $r \sim 1$. In particular, we have
\begin{equation*}
 |\Omega^k \phi| \sim |\nablaslash^k \phi|.
\end{equation*}

\subsection{Energy Estimates Scheme}\label{Energy estimates scheme}
Let $\phi$ be a solution for the following non-homogenous wave equation on $\mathbb{R}^{2+1}$:
\begin{equation}\label{non-homogenous wave equation}
 \Box \phi = \Phi.
\end{equation}
The energy momentum tensor associated to $\phi$ is
\begin{equation*}
 \mathbb{T}_{\alpha \beta}[\phi] = \nabla_\alpha \phi \nabla_\beta \phi -\frac{1}{2}g_{\alpha \beta} \nabla^{\mu} \phi \nabla_{\mu} \phi.
\end{equation*}
It is symmetric and satisfies the following divergence identity,
\begin{equation}\label{divergence of T}
 \nabla^\alpha  \mathbb{T}_{\alpha \beta}[\phi] = \Phi \cdot \nabla_\beta \phi.
\end{equation}

Given a vector field $X$, which is usually called a \emph{multiplier vector field}, the associated energy currents are defined as follows
\begin{equation*}
 J^{X}_{\alpha}[\phi] = \mathbb{T_{\alpha\mu}}[\phi]X^{\mu}, \quad K^X [\phi] = \mathbb{T}^{\mu\nu}[\phi]\, ^{(X)}\pi_{\mu \nu},
\end{equation*}
where the deformation tensor $^{(X)}\pi_{\mu \nu}$ is defined by
\begin{equation}
  ^{(X)}\pi_{\mu \nu} = \frac{1}{2} \mathcal{L}_X g_{\mu \nu} = \frac{1}{2}(\nabla_\mu X_{\nu} + \nabla_\nu X_{\mu}).
\end{equation}
Thanks to \eqref{divergence of T}, we have
\begin{equation}\label{divergence of J}
 \nabla^\alpha J^{X}_{\alpha}[\phi] =  K^X [\phi] + \Phi \cdot X \phi
\end{equation}

By null frames $\{e_1 =\nablaslash, e_2 =\Lb, e_3 =L\}$, we express $\mathbb{T_{\alpha\beta}}[\phi]$ as
\begin{equation*}
\mathbb{T}(L,L)[\phi] = |L \phi|^2, \quad \mathbb{T}(L,\Lb)[\phi] =
 |\nablaslash \phi|^2, \quad \text{and}  \quad \mathbb{T}(\Lb,\Lb)[\phi] = |\Lb \phi|^2.
\end{equation*}
This manifests the dominant energy condition for $\mathbb{T_{\alpha\beta}}[\phi]$.

We shall use $X = \Omega$, $L$ and $\Lb$ as multiplier vector fields, the corresponding deformation tensors and currents are computed as follows,
\begin{equation}\label{deformation tensor of Omega-L-Lb}
\begin{split}
  ^{(\Omega)}\pi &= 0, \quad ^{(L)}\pi = \frac{1}{r} \gslash, \quad ^{(\Lb)}\pi = -\frac{1}{r} \gslash, \\
  K^\Omega &= 0, \quad  K^L = \frac{1}{2r} (|\nablaslash \phi|^2 + L\phi \, \Lb \phi), \quad K^{\Lb} = -\frac{1}{2r}(|\nablaslash \phi|^2+L\phi \, \Lb \phi).
\end{split}
\end{equation}
where $\gslash$ is the restriction of the Minkowski metric $m$ to the circle $S_{\ub,u}$.

\begin{minipage}[!t]{0.3\textwidth}
\includegraphics[width=2.5in]{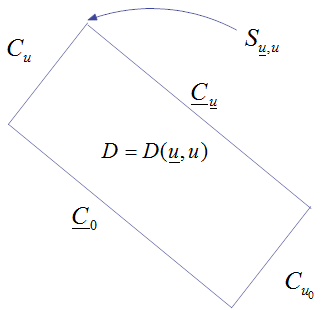}
\end{minipage}
\hspace{0.1\textwidth}
\begin{minipage}[!t]{0.5\textwidth}
We use $\mathcal{D}(u,\ub)$ to denote the space-time slab enclosed by the hypersurfaces $C_{u_0}$, $\Cb_0$, $C_{u}$ and $\Cb_{\ub}$. We integrate \eqref{divergence of J} on $\mathcal{D}(u,\ub)$, which is a domain enclosed by the null hypersurfaces $C_u, \Cb_{\ub}, C_{u_0}$ and $\Cb_0$, to derive
\begin{equation*}
\begin{split}
 &\quad \int_{C_u} \mathbb{T}[\phi](X,L)+\int_{\Cb_{\ub}} \mathbb{T}[\phi](X,\Lb)\\
 &=  \int_{C_{u_0}} \mathbb{T}[\phi](X,L)+\int_{\Cb_{0}} \mathbb{T}[\phi](X,\Lb)\\
 &\quad  +\doubleint_{\mathcal{D}(u,\ub)}  K^X [\phi] + \Phi \cdot X \phi.
 \end{split}
\end{equation*}
where  $L$ and $\Lb$ are corresponding normals of the null hypersurfaces $C_u$ and $\Cb_{\ub}$.
\end{minipage}

In applications, the data on $\Cb_{0}$ is always vanishing, thus, we have the following fundamental energy identity,
\begin{equation}\label{fundamental energy identity}
 \int_{C_u} \mathbb{T}[\phi](X,L)+\int_{\Cb_{\ub}} \mathbb{T}[\phi](X,\Lb)  =  \int_{C_{u_0}} \mathbb{T}[\phi](X,L) +\doubleint_{\mathcal{D}(u,\ub)}  K^X [\phi] + \Phi \cdot X \phi.
\end{equation}

\subsection{Null Forms}\label{section null form}
For a real valued quadratic form $Q$ defined on $\mathbb{R}^{2+1}$, it is called a \emph{null form} if for all null vector $\xi \in \mathbb{R}^{2+1}$, we have
$Q(\xi, \xi) = 0$. As an example, the metric tensor
\begin{equation}\label{basic null forms}
 Q_0 (\xi, \eta) = g(\xi, \eta),
\end{equation}
is a null form.

Given two scalar function $\phi$, $\psi$, we use $Q_0(\nabla \phi, \nabla \psi)$ to denote
\begin{equation*}
Q_0(\nabla \phi, \nabla \psi) = g^{\alpha\beta} \, \partial_\alpha \phi \, \partial_\beta \psi.
\end{equation*}

For the rotational vector field $\Omega$, we have
\begin{equation}\label{commute null form with Omega}
 \Omega( Q_0(\nabla \phi, \nabla \psi)) = Q_0(\nabla \Omega \phi, \nabla \psi) + Q_0(\nabla \phi, \nabla \Omega\psi).
\end{equation}

For a vector field $X$, we denote
\begin{equation*}
     (Q \circ X)(\nabla \phi, \nabla \psi) = Q(\nabla X \phi, \nabla \psi) + Q(\nabla \phi, \nabla X \psi),
\end{equation*}
and $[Q, X] = X Q- Q\circ X$, we then have
\begin{equation*}
\begin{split}
  [L, Q_0](\nabla \phi, \nabla \psi) & =\frac{1}{r}(2Q_0(\nabla \phi, \nabla \psi)+L\phi \Lb\psi+\Lb\phi L\psi), \\
   [\Lb, Q_0](\nabla \phi, \nabla \psi) & =-\frac{1}{r}(2Q_0(\nabla \phi, \nabla \psi)+L\phi \Lb\psi+\Lb\phi L\psi).
\end{split}
\end{equation*}

From the analytic point of view, we have the following point-wise estimates for $Q_0$ which will play a key role in the current work:
\begin{equation}\label{null form bound}
 |Q_0(\nabla \phi, \nabla \psi)| \lesssim |L \phi| \,|\Lb\psi| + |\Lb \phi|\,|L\psi| + |\nablaslash \phi|\, |\nablaslash \psi|.
\end{equation}
In particular, we see that the component $L\phi \cdot L \psi$ does not appear.

\subsection{Gronwall and Sobolev Inequalities} \label{Soboleve and Gronwall}
We first recall the standard Gronwall's inequality. Let $\phi(t)$ be a non-negative function defined on an interval $I$ with initial point $t_0$. If $\phi$ satisfies the following ordinary differential inequality,
\begin{equation*}
 \frac{d}{d t} \phi \leq a \cdot \phi + b,
\end{equation*}
where two non-negative functions  $a$, $b \in L^1(I)$, then for all $t \in I$, we have,
\begin{equation}\label{Gronwall}
 \phi(t) \leq e^{A(t)}(\phi(t_0)+\int_{t_0}^t e^{-A(\tau)}b(\tau)d\tau),
\end{equation}
where $A(t) = \int_{t_0}^t a(\tau) d\tau$.
\\

We the recall the Sobolev inequalities on the circle $S_{\ub,u}$.
\begin{equation}\label{L-infty into L-2}
    \sup_{S_{\ub, u}}|f|\leq |u|^{\frac{1}{2}}(\int_{S_{\ub, u}}|\nablaslash f|^2 d \mu_{\gslash})^{\frac{1}{2}}+|u|^{-\frac{1}{2}}(\int_{S_{\ub, u}}|f|^2 d \mu_{\gslash})^{\frac{1}{2}}.
\end{equation}
\begin{equation}\label{L-6 into L-4 and L-2}
    \int_{S_{\ub, u}}|f|^6 d \mu_{\gslash}\leq (\int_{S_{\ub, u}}|f|^4d\mu_{\gslash})\{|u|\int_{S_{\ub, u}}|\nablaslash f|^2d\mu_{\gslash}+|u|^{-1}\int_{S_{\ub, u}}|f|^2d\mu_{\gslash}\}
\end{equation}
where $\phi$ is a smooth function on $S_{\ub,u}$. We remark that, in sequel, $|u| \sim r$ (up to an error of size $\delta$) is the diameter of the sphere $S_{\ub,u}$.\\
Finally, we recall some Sobolev inequalities on null cones $C_u$ and $\Cb_{\ub}$. Since the proofs reflect the influence the geometry of the cones, we shall give the proofs in detail.
\begin{proposition}\label{sobolev on cu}
Let $\phi$ be a smooth function on $C_{u}$ vanishing on $S_{0,u}$, then we have
\begin{equation*}
 |u|^{\frac{1}{4}} \|\phi \|_{L^{4}(S_{\ub,u})} \lesssim \| L\phi \|^{\frac{1}{2}}_{L^{2}(C_{u})}(\| \phi \|^{\frac{1}{2}}_{L^{2}(C_{u})}
 +|u|^{\frac{1}{2}}\|\nablaslash\phi\|^{\frac{1}{2}}_{L^{2}(C_{u})}),
 \end{equation*}
\begin{equation*}
\|\phi \|_{L^{2}(S_{\ub,u})} \lesssim \| L\phi \|^{\frac{1}{2}}_{L^{2}(C_{u})}\| \phi \|^{\frac{1}{2}}_{L^{2}(C_{u})}.
 \end{equation*}
\end{proposition}

\begin{proof}
We multiply \eqref{L-6 into L-4 and L-2} by $|u|^2$ and integrate on $C_u$ to derive
\begin{equation}\label{L-6 into L-4 on sphere and L-2 on cu}
    \int_{C_{u}}|u|^2|\phi|^6 d \mu_{\gslash}\leq \sup_{\ub}(\int_{S_{\ub, u}}|u||\phi|^4d\mu_{\gslash})\{\int_{C_{u}}||u|\nablaslash \phi|^2+\int_{C_{u}}|\phi|^2\}.
\end{equation}
We consider the integral $ x(\ub)=\int_{S_{\ub, u}}|u||\phi|^4 d \mu_{\gslash}$ on $S_{\ub,u}$ (notice that $u$ is fixed). Its derivative in $\ub$ direction reads as
 \begin{equation*}
    \frac{d}{d \ub}x(\ub)=\int_{S_{\ub, u}}|u|\{L(|\phi|^4)+\frac{1}{r}|\phi|^4\} d \mu_{\gslash}.
\end{equation*}
We remark that the term $\dfrac{1}{r}$ comes from the derivative of the volume form and it is the null expansion (mean curvature) of $S_{\ub,u}$ in the $L$ direction. Therefore,
\begin{equation}\label{ode of x(ub)}
    \frac{d}{d \ub}x(\ub)\lesssim\frac{1}{|u|}x+b,
\end{equation}
where $b(\ub)=\int_{S_{\ub, u}} |u||\phi|^3|L \phi| d \mu_{\gslash}$. Recall that $x(0)=0$ since $\phi$ vanishes on $S_{0,u}$. In view of the fact that $\ub \leq \delta$ and $\delta \ll u$, Gronwall's inequality yields
 \begin{equation*}
   x(\ub) \lesssim\int_{C_{u}}4|u||\phi|^3|L \phi| \lesssim (\int_{C_{u}}|u|^2|\phi|^6)^{\frac{1}{2}}(\int_{C_{u}}|L \phi|^2)^{\frac{1}{2}}.
\end{equation*}
Thus, \eqref{ode of x(ub)} implies
\begin{equation}\label{L4 into L6 and L2 on cu}
    \sup_{\ub}(\int_{S_{\ub, u}}|u||\phi|^4d\mu_{\gslash})\lesssim (\int_{C_{u}}|u|^2|\phi|^6)^{\frac{1}{2}}(\int_{C_{u}}|L \phi|^2)^{\frac{1}{2}}
\end{equation}
We Substitute \eqref{L6 into L2 on cu} in \eqref{L-6 into L-4 on sphere and L-2 on cu} and we cancel $(\int_{C_{u}}|u|^2|\phi|^6)^{\frac{1}{2}}$ on both sides to obtain
\begin{equation}\label{L6 into L2 on cu}
    (\int_{C_{u}}|u|^2|\phi|^6)^{\frac{1}{2}}\lesssim (\int_{C_{u}}|L \phi|^2)^{\frac{1}{2}} \{\int_{C_{u}}||u|\nablaslash \phi|^2+\int_{C_{u}}|\phi|^2\}
\end{equation}
We then substitute \eqref{L6 into L2 on cu} in \eqref{L4 into L6 and L2 on cu} and this proves the first inequality in the proposition. The second inequality can be proved in a similar but much more straightforward way.
\end{proof}

\begin{proposition}\label{sobolev on cbub}
Let $\phi$ be a smooth function on $\Cb_{\ub}$, we have the following estimates:
\begin{equation*}
 |u|^{\frac{1}{4}} \|\phi \|_{L^{4}(S_{\ub,u})} \lesssim |u_0|^{\frac{1}{4}}\|\phi\|_{L^{4}(S_{\ub,u_{0}})}+ \| \Lb\phi \|^{\frac{1}{2}}_{L^{2}(\Cb_{\ub})}(\| \phi \|^{\frac{1}{2}}_{L^{2}(\Cb_{\ub})}
 +\||u'|\nablaslash\phi\|^{\frac{1}{2}}_{L^{2}(\Cb_{\ub})}).
 \end{equation*}
\begin{equation*}
\|\phi \|_{L^{2}(S_{\ub,u})} \lesssim \|\phi \|_{L^{2}(S_{\ub,u_0})}+ \| \Lb\phi \|^{\frac{1}{2}}_{L^{2}(\Cb_{\ub})}\| \phi \|^{\frac{1}{2}}_{L^{2}(\Cb_{\ub})}.
 \end{equation*}
\end{proposition}
\begin{proof}We multiply \eqref{L-6 into L-4 and L-2} by $|u|^2$ and integrate on $\Cb_{\ub}$ to derive
\begin{equation}\label{L-6 into L-4 on sphere and L-2 on cbub}
    \int_{\Cb_{\ub}}|u|^2|\phi|^6 d \mu_{\gslash}\leq \sup_{u}(\int_{S_{\ub, u}}|u||\phi|^4d\mu_{\gslash})\{\int_{\Cb_{\ub}}||u'|\nablaslash \phi|^2+\int_{\Cb_{\ub}}|\phi|^2\}.
\end{equation}
We consider the integral $x(u)=\int_{S_{\ub, u}}|u||\phi|^4 d \mu_{\gslash}$ on $S_{\ub, u}$ (notice that $\ub$ is fixed). Its derivative in $u$ direction reads as
 \begin{equation*}
    \frac{d}{d u}x(u)=\int_{S_{\ub, u}}|u|\{\Lb(|\phi|^4)+(-\frac{1}{r}-\frac{1}{|u|})|\phi|^4\} d \mu_{\gslash}.
\end{equation*}
We remark that the term $-\dfrac{1}{r}$ comes from the derivative of the volume form and it is the incoming null expansion (mean curvature) of $S_{\ub,u}$ in the $\Lb$ direction. Since $-\dfrac{1}{r}-\dfrac{1}{|u|}<0$, we have
 \begin{equation*}\label{ode of x(u)}
    \frac{d}{d u}x(u)\leq \int_{S_{\ub, u}}4|u||\phi|^3|\Lb \phi| d \mu_{\gslash},
\end{equation*}
We integrate on $\Cb_{\ub}$ to derive
 \begin{equation*}
    x(u)\lesssim x(u_0)+ (\int_{\Cb_{\ub}}|u'|^2|\phi|^6)^{\frac{1}{2}}(\int_{\Cb_{\ub}}|\Lb \phi|^2)^{\frac{1}{2}}.
\end{equation*}
Let $A=\int_{\Cb_{\ub}}|\Lb \phi|^2$ and $y=\int_{\Cb_{\ub}}|u|^2|\phi|^6$, we have
\begin{equation}\label{L4 into L6 and L2 on cbub}
    \sup_{u}(x(u))\lesssim x(u_0)+A^{\frac{1}{2}}y^{\frac{1}{2}}.
\end{equation}
Let $B=\int_{\Cb_{\ub}}||u'|\nablaslash \phi|^2+\int_{\Cb_{\ub}}|\phi|^2$. We substitute \eqref{L4 into L6 and L2 on cbub} in \eqref{L-6 into L-4 on sphere and L-2 on cbub}, therefore,
\begin{equation*}
    y\lesssim (x(u_0)+A^{\frac{1}{2}}y^{\frac{1}{2}})B.
\end{equation*}
This implies $y\lesssim (x(u_0)+AB)B$. Thus,
\begin{equation}\label{L6 into L2 on cbub-1}
    A^{\frac{1}{2}}y^{\frac{1}{2}}\lesssim (2x(u_0)AB)^{\frac{1}{2}}+AB\lesssim x(u_0)+2AB ,
\end{equation}
We now substitute \eqref{L6 into L2 on cbub-1} in turn in \eqref{L4 into L6 and L2 on cbub}. It leads to the first inequality of the proposition. The second inequality can be proved in a similar but much more straightforward way.
\end{proof}

Combing the above estimates, we have the following Sobolev estimates for $L^\infty$ norms:
\begin{equation*}
   \|\phi\|_{L^\infty} \lesssim |u|^{-\frac{1}{2}}(\| L\phi \|^{\frac{1}{2}}_{L^{2}(C_{u})}\| \phi \|^{\frac{1}{2}}_{L^{2}(C_{u})}+\||u|\nablaslash L\phi \|^{\frac{1}{2}}_{L^{2}(C_{u})}\||u| \nablaslash\phi \|^{\frac{1}{2}}_{L^{2}(C_{u})}),
\end{equation*}
and
\begin{equation*}
    \begin{split}
       \|\phi\|_{L^\infty} & \lesssim |u|^{\frac{1}{2}}\|\nablaslash\phi \|_{L^{2}(S_{\ub,u_0})}+|u|^{-\frac{1}{2}}\|\phi \|_{L^{2}(S_{\ub,u_0})}\\
         &+|u|^{-\frac{1}{2}}(\| \Lb\phi \|^{\frac{1}{2}}_{L^{2}(\Cb_{\ub})}\| \phi \|^{\frac{1}{2}}_{L^{2}(\Cb_{\ub})}+\||u'|\nablaslash \Lb\phi \|^{\frac{1}{2}}_{L^{2}(\Cb_{\ub})}\||u'| \nablaslash\phi \|^{\frac{1}{2}}_{L^{2}(\Cb_{\ub})}).
     \end{split}
\end{equation*}

\subsection{Strategy of the Proof}

We now sketch the main structures of the proof. The data will be eventually given on $t = u_0 + \delta$ and the solution will exits at least for $t \in [u_0 + \delta, -1]$. This can be read off from the following picture.

\includegraphics[width=4.5 in]{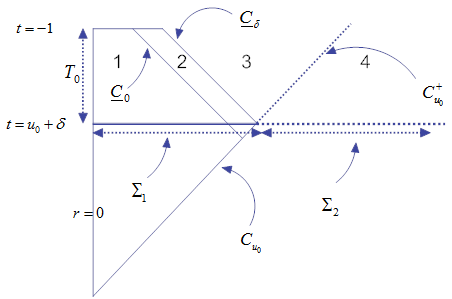}

\begin{itemize}
\item First Step. We give initial data on the null hypersurface $C_{u_0}$ where $u_0 \leq \ub \leq \delta$. When $u_0 \leq \ub \leq 0$, the data is trivial, therefore the solution in Region $1$ in the picture is a constant map. When $0 \leq \ub \leq \delta$, the data will be prescribed in a specific form (see next section for detailed account):
    \begin{equation*}
 \varphi(\ub, u_0, \theta) = \frac{\delta^{\frac{1}{2}}\psi_0 (\frac{\ub}{\delta}, \theta) + (0,0,1)}{\sqrt{\delta|\psi_0 (\frac{\ub}{\delta}, \theta)|^2 +1}}.
\end{equation*}
where the energy will be approximately $E_0$. We then show that we can construct a solution in Region $2$ in the picture. As a consequence, we take the restriction of the solution constructed in this step to the surface $\Sigma_1 \subset \{t = u_0 + \delta\}$ (see the above picture) as the first part of the Cauchy data.

\item Second Step. We choose a smooth extension of the data constructed from the first step to $\Sigma_2 \subset \{t = u_0 + \delta\}$ in such a way that the energy (up to certain order) is small. By small data theory, we can construct a solution in Region $4$.

\item Third Step. From previous two steps, we actually can show that the restriction of the solution already constructed to $\Cb_{\delta}$ and $C_{u_0}^+$ (where $\ub \geq \delta$) are small in energy norms. We use them as initial data and we can then solve another small data problem to construct solution on Region $3$ in the picture. We finally patches the solutions in Region $1,2,3$ and $4$ to finish the construction.
\end{itemize}

We remark that the first step is the most difficult part since the data is no longer small and we have to carefully deal with the cancelations from the structure of the wave map equations and the profile of the data. The second and the third parts are more or less standard.

\section{Choice of the Initial Data in Region $1$ and $2$}\label{Section Initial Data in Short Pulse Regime}

Let $\widetilde{C}_{u_0}$ be the truncated light-cone defined as follows
\begin{equation*}
\widetilde{C}_{u_0} = \{x \in \mathbb{R}^{3+1} | u(x) = u_0,  u_0 \leq \ub(x) \leq \delta \}.
\end{equation*}

First of all, we require that the data $\varphi(\ub,u_0,\theta)$ (where $u_0$ now is regarded as a fixed value) of \eqref{Main Equation} to satisfy
\begin{equation*}
\varphi(\ub,u_0,\theta) = (0,0,1) \in \mathbb{S}^2 \subset \mathbb{R}^3 \quad \text{for all} \quad \ub \leq 0.
\end{equation*}
Therefore, according to the weak Huygens principle for wave equations, the solution $\varphi$ of \eqref{Main Equation} satisfies
 \begin{equation*}
 \varphi(x) \equiv (0,0,1)  \quad \text{in Region $1$}, \,\, i.e., \quad \ub(x) \leq 0,  u_0 \leq u(x) \leq 0.
 \end{equation*}
 In particular, $\varphi \equiv (0,0,1)$ on $\Cb_0$ up to infinite order.

Secondly, we prescribe $\varphi$ on $C_{u_0}$ by
\begin{equation}\label{precise short pulse data on C_0}
 \varphi(\ub, u_0, \theta) = \frac{\delta^{\frac{1}{2}}\psi_0 (\frac{\ub}{\delta}, \theta) + (0,0,1)}{\sqrt{\delta|\psi_0 (\frac{\ub}{\delta}, \theta)|^2 +1}},
\end{equation}
where $\psi_0$ is a \emph{fixed} smooth $\mathbb{R}^3$-valued function supported in $(0,1)$.

The data given in the above form is called a \emph{short pulse}, a name invented by Christodoulou in \cite{Ch-08}. In his work, he prescribed the shear (more precisely, the conformal geometry) of the initial null hypersurface in a similar form as \eqref{precise short pulse data on C_0}. The shear in the situation of \cite{Ch-08} is exactly the initial data for Einstein vacuum equation.

We remark that the derivative of the data can be extremely \emph{large} if $\delta$ is small. It will be obvious once we derive $L^\infty$ estimates for derivatives of $\varphi$ in the rest of the section.

In order to derive estimates for initial data, we collect some commutator formulas for future use. Denote Lie derivatives $D=\mathcal{L}_L, \underline{D}=\mathcal{L}_{\underline{L}}$, on $\mathbb{R}^{2+1}$ we have,
\begin{equation}\label{commutates}
\begin{split}
[\mathcal{L}_\Omega, \nablaslash]&=0, \quad [D, \nablaslash]=0, \quad [\underline{D}, \nablaslash]=0,\\
[\Box, \Omega]&=0, \quad [D, \Omega]=0, \quad [\underline{D}, \Omega]=0,\\
[\Box, L]&=\frac{1}{2r^2}(L-\Lb)+\frac{2}{r}\laplacianslash, \,\,\,[\Box, \Lb]=\frac{1}{2r^2}(\Lb-L)-\frac{2}{r}\laplacianslash.
\end{split}
\end{equation}
Commuting $\Omega$ with \eqref{Main Equation} $n$ times, in view of \eqref{commutates}, we have\footnote{ We shall ignore the numerical constants since they are irrelevant in the context.}
\begin{equation}\label{commute n Omega with main equation}
\begin{split}
 \Box \Omega^{n} \varphi &= \sum_{ i+j=n}C_{n}^{i}\Omega^{i}\varphi\sum_{ p+q=j}
Q_{0}(\nabla\Omega^{p}\varphi, \nabla\Omega^{q}\varphi)\\
&= \sum_{ i+p+q=n} \Omega^{i}\varphi \cdot Q_{0}(\nabla\Omega^{p}\varphi, \nabla\Omega^{q}\varphi).
\end{split}
\end{equation}
Commuting $L, \Omega$ with \eqref{Main Equation}, in view of \eqref{commutates}, we derive
\begin{equation*}\label{commute L Omega n with main equation}
\begin{split}
 \Box L \Omega^{n} \varphi &= \sum_{ i+j=n} \{C_{n}^{i}L\Omega^{i}\varphi\sum_{ p+q=j}
Q_{0}(\nabla\Omega^{p}\varphi, \nabla\Omega^{q}\varphi)+2C_{n}^{i}\Omega^{i}\varphi\sum_{ p+q=j}
Q_{0}(\nabla L\Omega^{p}\varphi, \nabla\Omega^{q}\varphi) \}\\
 & + \sum_{ i+j=n}2C_{n}^{i}\Omega^{i}\varphi\sum_{ p+q=j}
\frac{1}{r}\{2Q_{0}(\nabla \Omega^{p}\varphi, \nabla\Omega^{q}\varphi)+L\Omega^{p}\varphi \Lb\Omega^{q}\varphi+\Lb\Omega^{p}\varphi L\Omega^{q}\varphi\}\\
&- \frac{1}{2r^2}(\Lb \Omega^{n} \varphi -L \Omega^{n} \varphi)+\frac{2}{r}\laplacianslash \Omega^{n} \varphi.
\end{split}
\end{equation*}
Thus,
\begin{equation}\label{commute L Omega n with main equation}
\begin{split}
 \Box L \Omega^{n} \varphi &= \sum_{ i+p+q=n} \{L\Omega^{i}\varphi\cdot
Q_{0}(\nabla\Omega^{p}\varphi, \nabla\Omega^{q}\varphi)+\Omega^{i}\varphi \cdot Q_{0}(\nabla L\Omega^{p}\varphi, \nabla\Omega^{q}\varphi) \}\\
 & + \sum_{ i+p+q=n}\Omega^{i}\varphi \cdot \frac{1}{r}\{Q_{0}(\nabla \Omega^{p}\varphi, \nabla\Omega^{q}\varphi)+L\Omega^{p}\varphi \Lb\Omega^{q}\varphi+\Lb\Omega^{p}\varphi L\Omega^{q}\varphi\}\\
&- \frac{1}{2r^2}(\Lb \Omega^{n} \varphi -L \Omega^{n} \varphi)+\frac{2}{r}\laplacianslash \Omega^{n} \varphi.
\end{split}
\end{equation}
Similarly, we commute $\Lb, \Omega$ with \eqref{Main Equation} to derive
\begin{equation}\label{commute Lb Omega n with main equation}
\begin{split}
 \Box \Lb \Omega^{n} \varphi &= \sum_{ i+p+q=n}\{\Lb\Omega^{i}\varphi \cdot Q_{0}(\nabla\Omega^{p}\varphi, \nabla\Omega^{q}\varphi)+ \Omega^{i}\varphi \cdot Q_{0}(\nabla \Lb\Omega^{p}\varphi, \nabla\Omega^{q}\varphi)\}\\
& +\sum_{ i+p+q=n}\Omega^{i}\varphi \cdot \frac{1}{r}\{Q_{0}(\nabla \Omega^{p}\varphi, \nabla\Omega^{q}\varphi)+L\Omega^{p}\varphi \Lb\Omega^{q}\varphi+\Lb\Omega^{p}\varphi L\Omega^{q}\varphi\}\\
&+ \frac{1}{2r^2}(\Lb \Omega^{n} \varphi -L \Omega^{n} \varphi)-\frac{2}{r}\laplacianslash \Omega^{n} \varphi.
\end{split}
\end{equation}

In view of \eqref{precise short pulse data on C_0}, by taking derivatives in $L$ or $\nablaslash$ direction, we obtain
\begin{equation*}
 \|L\varphi\|_{L^{\infty}(C_{u_{0}})} \lesssim \delta^{-\frac{1}{2}},
\end{equation*}
and
\begin{equation*}
\|\nablaslash\varphi\|_{L^{\infty}(C_{u_{0}})} \lesssim \delta^{\frac{1}{2}}.
\end{equation*}
In fact, by taking $L$ or $\nablaslash$ derivatives consecutively, for $k \in \mathbb{Z}_{\geq 0}$, we derive immediately,
\begin{equation}\label{L infinity on C0 for L or nablaslash directions}
\begin{split}
\|L\nablaslash^{k}\varphi\|_{L^{\infty}(C_{u_{0}})}
 &\lesssim_k \delta^{-\frac{1}{2}},\\
\|\nablaslash^{k+1}\varphi\|_{L^{\infty}(C_{u_{0}})} &\lesssim_k \delta^{\frac{1}{2}},\\
\|L^{2}\nablaslash^{k}\varphi\|_{L^{\infty}(C_{u_{0}})} &\lesssim_k \delta^{-\frac{3}{2}}.
\end{split}
\end{equation}
We turn to the $L^\infty$ estimates on those derivatives of $\varphi$ involving $\Lb$ directions and this process relies on the original equation \eqref{Main Equation}. We first rewrite \eqref{Main Equation} in terms of null frame:
\begin{equation}\label{Main Equation in null frame}
 -L \Lb \varphi+\laplacianslash \varphi+\frac{1}{2r}(L\varphi-\Lb\varphi)
=\varphi(-\Lb\varphi \cdot L\varphi + |\nablaslash \varphi|^2).
\end{equation}

We first derive estimates on $\Lb \varphi$. Notice that \eqref{Main Equation in null frame} can be viewed as an ordinary differential equation (ODE) for $\Lb \varphi$ as follows, \footnote{ Since the exact numerical constants are irrelevant, we shall ignore the constants appearing as coefficients.}
\begin{equation*}
 L (\Lb \varphi) = a \cdot \Lb \varphi+ b,
\end{equation*}
where
\begin{equation*}
 a = -(\frac{1}{2r}- \varphi L\varphi),\,\,\,
  b  = \frac{1}{2r}L\varphi +\laplacianslash \varphi-\varphi |\nablaslash\varphi|^2.
\end{equation*}
According to \eqref{L infinity on C0 for L or nablaslash directions}, we derive immediately,
\begin{equation*}
\|a\|_{L^{\infty}(C_{u_{0}})}  \lesssim \delta^{-\frac{1}{2}}, \,\,
\|b\|_{L^{\infty}(C_{u_{0}})}  \lesssim \delta^{-\frac{1}{2}}.
\end{equation*}
Observe that
\begin{equation*}
L |\Lb \varphi|\leq| L (\Lb \varphi)| \leq|a| \cdot |\Lb \varphi|+ |b|.
\end{equation*}
In view of the fact that $\Lb \varphi \equiv 0$ on $S_{\ub,u}$, by Gronwall's inequality, we have
\begin{equation*}
 |\Lb \varphi (u)| \lesssim \int_0^\delta e^{-A(\tau)} b(\tau)d\tau \lesssim \delta^{\frac{1}{2}}.
\end{equation*}
Hence,
 \begin{equation}\label{L infinity on Lb phi on C_0}
 \|\Lb\varphi\|_{L^{\infty}(C_{u_{0}})} \lesssim \delta^{\frac{1}{2}}.
 \end{equation}

We proceed to the estimates on $\Lb \nablaslash \varphi$. We shall first commute \eqref{Main Equation} with $\Omega$, that is, take $n=1$ in \eqref{commute n Omega with main equation}. In null frame, we rewrite the equation as
 \begin{equation*}
-L\Lb\Omega\varphi+\laplacianslash\Omega \varphi +\frac{1}{2r}(L\Omega\varphi-\Lb\Omega\varphi) = 2 \varphi Q_{0}(\nabla\Omega\varphi, \nabla\varphi)+ \Omega\varphi Q_{0}(\nabla\varphi, \nabla\varphi).
 \end{equation*}
Similarly as above, by Gronwall's inequality, we obtain
$\|\Lb\Omega\varphi\|_{L^{\infty}(C_{u_{0}})}\lesssim \delta^{\frac{1}{2}}.$
Hence, according to \eqref{Compare Omega and nablaslash},
\begin{equation}\label{L infinity on Lb nablaslash phi on C_0}
\|\Lb\nablaslash\varphi\|_{L^{\infty}(C_{u_{0}})}\lesssim \delta^{\frac{1}{2}}.
 \end{equation}

Similarly, we commute \eqref{Main Equation} with two $\Omega$'s to derive
\begin{equation}\label{L infinity on Lb nablaslash nablaslash phi on C_0}
\|\Lb\nablaslash^2 \varphi\|_{L^{\infty}(C_{u_{0}})}\lesssim \delta^{\frac{1}{2}}.
 \end{equation}

\begin{remark}[Key Ingredient]
To obtain a long time existence theorem for \eqref{Main Equation}, we have to derive estimates on $\varphi$ (as well as on its derivatives). Those estimates must be valid on the initial hypersurface and those estimates should propagate along the evolution of \eqref{Main Equation}. For this purpose, we shall use a slightly weaker estimate for $\nablaslash^k \varphi$ than those in \eqref{L infinity on C0 for L or nablaslash directions},
\begin{equation*}
  \|\nablaslash^{k+1}\varphi\|_{L^{\infty}(C_{u_{0}})} \lesssim_k 1.
\end{equation*}
One expects it should be much easier to prove this estimate propagating along the flow of \eqref{Main Equation} than the original one in \eqref{L infinity on C0 for L or nablaslash directions}. This is the relaxation mentioned in the introduction.
\end{remark}

To summarize, on the initial null hypersurface $C_{u_0}$, under the initial assumption \eqref{precise short pulse data on C_0}, up to third derivatives of $\varphi$ (which turns to be the minimal number of derivatives needed to proceed an bootstrap argument in next section), we have the following $L^\infty$ estimates,
\begin{equation}\label{L infinity on initial surface 1}
\begin{split}
 \|L\nablaslash^{k} \varphi\|_{L^{\infty}(C_{u_{0}})}  &\lesssim \delta^{-\frac{1}{2}},\\
 \|\nablaslash^{k+1} \varphi\|_{L^{\infty}(C_{u_{0}})} &\lesssim 1,\\
 \|\Lb\nablaslash^{k}\varphi\|_{L^{\infty}(C_{u_{0}})} &\lesssim \delta^{\frac{1}{2}}.
\end{split}
\end{equation}
for $k=0, 1, 2,$ and
\begin{equation}\label{L infinity on initial surface 2}
\begin{split}
 \|L^{2}\varphi\|_{L^{\infty}(C_{u_{0}})} &\lesssim \delta^{-\frac{3}{2}}, \\
 \|L^{2}\nablaslash \varphi\|_{L^{\infty}(C_{u_{0}})} &\lesssim \delta^{-\frac{3}{2}},\\
 \|L^{2}\nablaslash^{2} \varphi\|_{L^{\infty}(C_{u_{0}})} &\lesssim \delta^{-\frac{3}{2}}.
 \end{split}
\end{equation}
From \eqref{L infinity on initial surface 1} and \eqref{L infinity on initial surface 2}, we derive immediately $L^2$ estimates (observe that the area of $C_{u_0}$ is comparable to $\delta$):
\begin{equation}\label{L2 on initial surface 1}
\begin{split}
 \|L\nablaslash^{k} \varphi\|_{L^2(C_{u_{0}})}  &\lesssim 1, \\
 \|\nablaslash^{k+1} \varphi\|_{L^{2}(C_{u_{0}})}  &\lesssim \delta^{\frac{1}{2}}.
 \end{split}
\end{equation}
for $k=0, 1, 2,$ and
\begin{equation}\label{L2 on initial surface 2}
\begin{split}
 \|L^{2}\varphi\|_{L^{2}(C_{u_{0}})} &\lesssim \delta^{-1}, \\
 \|L^{2}\nablaslash \varphi\|_{L^{2}(C_{u_{0}})} &\lesssim \delta^{-1},\\
 \|L^{2}\nablaslash^{2} \varphi\|_{L^{2}(C_{u_{0}})} &\lesssim \delta^{-1}.
\end{split}
\end{equation}

In the next section, we shall show that, up to a universal constant, the estimates in \eqref{L2 on initial surface 1} and \eqref{L2 on initial surface 2}  will hold on all later outgoing null hypersurfaces $C_u$ where $-1> u > u_0$ provided the solution of \eqref{Main Equation} can be constructed up to $C_u$.

\section{A Priori Estimates}\label{Section A priori Estimates}
We assume that there is a solution of \eqref{Main Equation} defined on the domain $\mathcal{D}_{u,\ub}$ which is enclosed by the null hypersurfaces $C_{u}$, $\Cb_{\ub}$, $C_{u_0}$ and $\Cb_0$. The goal is to show that estimates \eqref{L2 on initial surface 1} and \eqref{L2 on initial surface 2}, which are valid on $C_{u_0}$, also hold on later $C_u$. We start by defining a family of energy norms. For this purpose, we slightly abuse the notations: we use $C_u$ to denote $C^{[0,\ub]}_u$ and $\Cb_{\ub}$ to denote $\Cb^{[u_0,u]}_{\ub}$, by definition,
\begin{equation*}
C^{[0,\ub]}_u = \{ p \in C_u | 0 \leq \ub(p) \leq \ub\} \quad \text{and} \quad \Cb^{[u_0,u]}_{\ub} = \{ p \in \Cb_{\ub} | u_0 \leq u(p) \leq u\}.
\end{equation*}
We define
\begin{equation}
\begin{split}
E_1(u,\ub) &= \|L \varphi\|_{L^2(C_u)} + \delta^{-\frac{1}{2}}\|\nablaslash \varphi\|_{L^2(C_u)},\\
\Eb_1(u, \ub) &=\|\nablaslash \varphi\|_{L^2(\Cb_{\ub})} + \delta^{-\frac{1}{2}} \|\Lb \varphi\|_{L^2(\Cb_{\ub})},\\
E_2(u,\ub) &=\|L \nablaslash \varphi\|_{L^2(C_u)} + \delta^{-\frac{1}{2}} \|\nablaslash^2 \varphi\|_{L^2(C_u)},\\
\Eb_2(u, \ub) &=  \|\nablaslash^2 \varphi\|_{L^2(\Cb_{\ub})} + \delta^{-\frac{1}{2}} \|\Lb \nablaslash \varphi\|_{L^2(\Cb_{\ub})},\\
E_3(u,\ub) &= \|L \nablaslash^2 \varphi\|_{L^2(C_u)} + \delta^{-\frac{1}{2}} \|\nablaslash^3 \varphi\|_{L^2(C_u)},\\
\Eb_3(u, \ub) &= \| \nablaslash^3 \varphi\|_{L^2(\Cb_{\ub})} + \delta^{-\frac{1}{2}}\|\Lb \nablaslash^2 \varphi\|_{L^2(\Cb_{\ub})}.
\end{split}
\end{equation}
and remember that $|\Omega \varphi| \sim |\nablaslash \varphi|$ on $\mathcal{D}_{u,\ub}$.

We also need another family of norms which involves at least two null derivatives. They are defined as follows:
\begin{equation}
\begin{split}
F_2(u,\ub) &=\delta\|L^2 \varphi\|_{L^2(C_u)}, \\
\Fb_2(u, \ub) &=\|\Lb^2 \varphi\|_{L^2(\Cb_{\ub})},\\
F_3(u,\ub) &=\delta \|L^2 \nablaslash \varphi\|_{L^2(C_u)},\\
\Fb_3(u, \ub) &= \|\Lb^2 \nablaslash \varphi\|_{L^2(\Cb_{\ub})}.
\end{split}
\end{equation}
We shall prove the following propagation estimates,
\begin{Main A priori Estimates}\label{a priori estimates} If $\delta$ is sufficiently small, for all initial data of \eqref{Main Equation} and all $I \in \mathbb{R}_{>0}$ which satisfy
\begin{equation}\label{initial bound}
E_1(u_0,\delta) + E_2(u_0,\delta)+E_3(u_0,\delta)+ F_2(u_0,\delta) + F_3(u_0,\delta)\leq I,
\end{equation}
there is a constant $C(I)$ depending only on $I_3$ (in particular, not on $\delta$), so that
\begin{equation}\label{main estimates}
\sum_{i=1}^3 [E_i(u,\ub) + \Eb_i(u, \ub)]+\sum_{j=2}^3 [F_j(u,\ub) + \Fb_j(u, \ub)] \leq C(I),
\end{equation}
for all $u \in [u_0, u^*]$ and $\ub \in [0, \ub^*]$ where $u_0 \leq u^* \leq -1$ and $0\leq \ub^* \leq \delta$.
\end{Main A priori Estimates}

\subsection{Bootstrap Assumption}
We shall use a standard bootstrap argument to prove the \textbf{Main A priori Estimates}. We assume that
\begin{equation}\label{bootstrap assumption}
\sum_{i=1}^3 [E_i(u,\ub) + \Eb_i(u, \ub)]+\sum_{j=2}^3 [F_j(u,\ub) + \Fb_j(u, \ub)] \leq M,
\end{equation}
for all $u \in [u_0, u^*]$ and $\ub \in [0, \ub^*]$, where $M$ is a sufficiently large constant. Since we have assumed the existence of the solution up to  $C_{u^*}$ and $\Cb_{\ub^*}$, we can always choose such a $M$ which may depend on the $\varphi$. At the end of the current section, we will show that we can actually choose $M$ so that it depends only on the norm of the initial data but not the profile. This will yield the \textbf{Main A priori Estimates}.

\subsection{Preliminary Estimates}\label{section Preliminary Estimates}
Under the bootstrap assumption \eqref{bootstrap assumption}, we first derive $L^\infty$ for one derivatives of $\varphi$. As a byproduct, we will also obtain the $L^4$ estimates for two derivatives of $\varphi$. For this purpose, we shall repeatedly use the Sobolev inequalities stated in Section \ref{Soboleve and Gronwall}.

We start with $L \varphi$. According to Sobolev inequalities, we have
\begin{align*}
 |u|^{\frac{1}{4}}\| L\varphi \|_{L^{4}(S_{\ub,u})} &\lesssim  \| L^{2}\varphi\|^{\frac{1}{2}}_{L^{2}(C_{u})}(\| L\varphi \|^{\frac{1}{2}}_{L^{2}(C_{u})}
 +|u|^{\frac{1}{2}}\| L\nablaslash\varphi \|_{L^{2}(C_{u})})\\
 &\lesssim (\delta^{-1} M )^\frac{1}{2}(M^\frac{1}{2}+ M^\frac{1}{2}).
 \end{align*}
Hence,
\begin{equation}\label{e 1}
\| L\varphi \|_{L^{4}(S_{\ub,u})}\lesssim \delta^{-\frac{1}{2}} M.
\end{equation}
Similarly, we have
\begin{equation}\label{e_2}
\| L\nablaslash \varphi \|_{L^{4}(S_{\ub,u})} \lesssim \delta^{-\frac{1}{2}} M.
\end{equation}
We also have
\begin{equation*}
\begin{split}
   \|L\varphi\|_{L^\infty(S_{\ub, u})} &\lesssim \| L^2\varphi \|^{\frac{1}{2}}_{L^{2}(C_{u})}\| L\varphi \|^{\frac{1}{2}}_{L^{2}(C_{u})}+\|\nablaslash L^2\varphi \|^{\frac{1}{2}}_{L^{2}(C_{u})}\| \nablaslash L\varphi \|^{\frac{1}{2}}_{L^{2}(C_{u})})\\
&\lesssim \delta^{-\frac{1}{2}}M.
\end{split}
\end{equation*}
i.e.
\begin{equation}\label{e_3}
 \|L\varphi\|_{L^\infty}\lesssim  \delta^{-\frac{1}{2}} M.
 \end{equation}

We now treat $\nablaslash \varphi$. According to Sobolev inequalities, we have
\begin{equation*}
\begin{split}
 \| \nablaslash \varphi \|_{L^{4}(S_{\ub,u})} &\lesssim  \| L \nablaslash\varphi\|^{\frac{1}{2}}_{L^{2}(C_{u})}(\| \nablaslash \varphi \|^{\frac{1}{2}}_{L^{2}(C_{u})}
 +\| \nablaslash^2 \varphi \|_{L^{2}(C_{u})})\\
 &\lesssim M^\frac{1}{2}((\delta^{\frac{1}{2}}M)^\frac{1}{2}+ (\delta^{\frac{1}{2}}M)^\frac{1}{2}).
\end{split}
\end{equation*}
Thus,
\begin{equation}\label{e_4}
\| \nablaslash\varphi\|_{L^{4}(S_{\ub,u})}\lesssim \delta^{\frac{1}{4}} M.
\end{equation}
Similarly,
\begin{equation}\label{e_5}
\| \nablaslash^2 \varphi\|_{L^{4}(S_{\ub,u})}\lesssim \delta^{\frac{1}{4}} M.
\end{equation}
And
\begin{equation*}
   \|\nablaslash\varphi\|_{L^\infty(S_{\ub, u})} \lesssim \| \nablaslash L\varphi \|^{\frac{1}{2}}_{L^{2}(C_{u})}\| \nablaslash\varphi \|^{\frac{1}{2}}_{L^{2}(C_{u})}+\|\nablaslash^2 L\varphi \|^{\frac{1}{2}}_{L^{2}(C_{u})}\|\nablaslash^2\varphi \|^{\frac{1}{2}}_{L^{2}(C_{u})}.
\end{equation*}
gives
\begin{equation}\label{e_6}
 \|\nablaslash \varphi\|_{L^\infty} \lesssim \delta^{\frac{1}{4}} M.
 \end{equation}

Finally, we turn to the estimates on $\Lb\varphi$. For $\|\Lb \varphi\|_{L^4}$, we have
\begin{equation*}
\begin{split}
 \|\Lb\varphi \|_{L^{4}(S_{\ub,u})} &\lesssim \|\Lb\varphi\|_{L^{4}(S_{\ub,u_{0}})}+ \| \Lb^2\varphi \|^{\frac{1}{2}}_{L^{2}(\Cb_{\ub})}(\| \Lb\varphi \|^{\frac{1}{2}}_{L^{2}(\Cb_{\ub})}
 +\|\nablaslash\Lb\varphi\|^{\frac{1}{2}}_{L^{2}(\Cb_{\ub})})\\
&\lesssim \delta^{\frac{1}{2}} + \delta^{\frac{1}{4}} M
\end{split}
\end{equation*}
If $\delta$ is sufficiently small, we obtain
\begin{equation}\label{e_7}
\| \Lb\varphi \|_{L^{4}(S_{\ub,u})} \lesssim \delta^{\frac{1}{4}} M.
\end{equation}
Similarly, we also obtain
\begin{equation}\label{e_8}
\| \Lb\nablaslash\varphi \|_{L^{4}(S_{\ub,u})}\lesssim \delta^{\frac{1}{4}} M.
\end{equation}
And
\begin{equation*}
    \begin{split}
       \|\Lb\varphi\|_{L^\infty(S_{\ub, u})} & \lesssim \|\nablaslash\Lb\varphi \|_{L^{2}(S_{\ub,u_0})}+\|\Lb\varphi \|_{L^{2}(S_{\ub,u_0})}\\
         &+\| \Lb^2\varphi \|^{\frac{1}{2}}_{L^{2}(\Cb_{\ub})}\| \Lb\varphi \|^{\frac{1}{2}}_{L^{2}(\Cb_{\ub})}+\|\nablaslash \Lb^2\varphi \|^{\frac{1}{2}}_{L^{2}(\Cb_{\ub})}\| \nablaslash\Lb\varphi \|^{\frac{1}{2}}_{L^{2}(\Cb_{\ub})}\\
&\lesssim \delta^{\frac{1}{2}} + \delta^{\frac{1}{2}} + \delta^{\frac{1}{4}}M.
     \end{split}
\end{equation*}
If $\delta$ is sufficiently small, we obtain
\begin{equation}\label{e_9}
\| \Lb\varphi\|_{L^{\infty}} \lesssim \delta^{\frac{1}{4}} M.
\end{equation}
We summarize all the estimates in the following proposition.
\begin{proposition}\label{proposition L infinity and L4 estimates} Under the bootstrap assumption \eqref{bootstrap assumption}, if $\delta$ is sufficiently small, we have
\begin{equation*}
\begin{split}
 &\quad \delta^{\frac{1}{2}} \|L\varphi\|_{L^\infty} + \delta^{-\frac{1}{4}} \|\nablaslash\varphi\|_{L^\infty} +  \delta^{-\frac{1}{4}} \|\Lb\varphi\|_{L^\infty}\\
&+\delta^{\frac{1}{2}} \| L\nablaslash \varphi \|_{L^{4}(S_{\ub,u})} +\delta^{-\frac{1}{4}} \|\nablaslash^2 \varphi\|_{L^{4}(S_{\ub,u})} + \delta^{-\frac{1}{4}} \|\Lb\nablaslash\varphi \|_{L^{4}(S_{\ub,u})}\\
&+\delta^{\frac{1}{2}} \| L\varphi \|_{L^{4}(S_{\ub,u})} +\delta^{-\frac{1}{4}}  \|\nablaslash \varphi\|_{L^{4}(S_{\ub,u})} + \delta^{-\frac{1}{4}} \|\Lb\varphi \|_{L^{4}(S_{\ub,u})}\lesssim M.
\end{split}
\end{equation*}
\end{proposition}

We observe that $L^\infty$ estimates on $\Lb \varphi$ (which of order $\delta^{\frac{1}{4}}$) is certainly worse than the initial estimates of $\Lb\varphi$ on $C_{u_0}$ (which is of order $\delta^{\frac{1}{2}}$). To rectify this loss in the future, we need $L^2$ estimates of $\Lb \varphi$ on $C_u$ (instead of $\Cb_{\ub}$ appearing in the definition of $\Eb_1(u, \ub)$).

\begin{lemma}\label{lemma L2 of Lb phi}
 Under the bootstrap assumption \eqref{bootstrap assumption}, if $\delta$ is sufficiently small, for $i=0$ or $1$ (not for $i=2$), we have
 \begin{equation*}
\|\Lb\Omega^i\varphi\|_{L^2(C_u)} \lesssim \delta M.
\end{equation*}
\end{lemma}
\begin{proof} We multiply $\Lb\varphi$ on both side of the main equation \eqref{Main Equation in null frame} and integrate on $C_{u}$. In view of the fact that $\Lb\varphi \equiv 0$ on $S_{0, u}$, \eqref{null form bound} as well as $|\varphi|\equiv 1$, this leads to
 \begin{equation}\label{aa}
\begin{split}
\int_{S_{\ub,u}} \!\! |\Lb \varphi|^2 &= \int_{C_{u}^{\ub}}\! L |\underline{L}\varphi|^{2} + \frac{1}{r}|\Lb \varphi|^2\\
&\lesssim \int_{C_{u}^{\underline{u}}}\frac{1}{r}|L\varphi||\underline{L}\varphi|+ |\laplacianslash\varphi||\underline{L}\varphi| + |L\varphi||\underline{L}\varphi|^{2} + |\nablaslash\varphi|^2|\underline{L}\varphi|.
\end{split}
\end{equation}
where the integral $ \int_{C_{u}^{\underline{u}}}$ means $\int_{0}^{\ub}\int_{S_{\ub',u}} d\ub'$. Let
$f^2(\ub) = \int_{C_{u}^{\underline{u}}}(\underline{L}\varphi)^{2}.$
We now estimate the terms at the right hand side of \eqref{aa} one by one:
\begin{equation*}
\begin{split}
\int_{C_{u}^{\underline{u}}}\frac{1}{r}|L\varphi||\underline{L}\varphi| &\lesssim f(\ub)M, \\
\int_{C_{u}^{\underline{u}}}|\laplacianslash\varphi||\underline{L}\varphi| &\lesssim
 \delta^{\frac{1}{2}}f(\ub)M,\\
 \int_{C_{u}^{\underline{u}}}|L\varphi||\underline{L}\varphi|^{2} &\lesssim  \delta^{-\frac{1}{2}}f^2(\ub) M,\\
 \int_{C_{u}^{\underline{u}}}|\nablaslash\varphi|^2|\underline{L}\varphi| &\lesssim \delta^{\frac{3}{4}} f(\ub) M.
\end{split}
\end{equation*}
Back to \eqref{aa}, we have
\begin{equation*}
\frac{d}{d\underline{u}}f(\ub)^{2}\lesssim M(\delta^{-\frac{1}{2}}f(\ub)^{2}+f(\ub)),
 \end{equation*}
We then integrate on $[0,\delta]$ to derive
\begin{equation*}
f(\ub) \lesssim M \delta^{\frac{1}{2}}f(\ub)+\delta M.
 \end{equation*}
If $\delta$ is sufficiently small, the first term on the right hand side will be absorbed. This yields the desire estimate. For $\Lb\Omega\varphi$, we simply repeat the above process to get the result.
\end{proof}

\subsection{Estimates on $E_{k}$'s and $\underline{E}_{k}$'s Part-1}
We commute $\Omega^i$ (for $i=1,2$) with \eqref{Main Equation}, in view of \eqref{commute n Omega with main equation}, we have
\begin{equation*}
 \Box \Omega^{i} \varphi = \sum_{k+p+q=i} \Omega^{k}\varphi \cdot Q_{0}(\nabla\Omega^{p}\varphi, \nabla\Omega^{q}\varphi).
\end{equation*}

 We use the scheme in Section \ref{Energy estimates scheme} for this equation where we take $\phi = \Omega^i \varphi$ ($i=0, 1, 2$) and $X = L$. In view of \eqref{fundamental energy identity}, we have
\begin{equation}\label{ES-E1-E2-E3-a}
\begin{split}
\int_{C_{u}}|L \Omega^i \varphi|^{2} +\int_{\underline{C}_{\underline{u}}}|\nablaslash \Omega^i \varphi|^{2} &=\int_{C_{u_{0}}}|L\Omega^i \varphi|^{2}+\sum_{p+q=i}\doubleint_{\mathcal{D}} \varphi Q_{0}(\nabla\Omega^{p}\phi, \Omega^{q}\nabla\phi)L \Omega^i \varphi \\
&\quad + \sum_{ k+p+q=i, k\geq 1} \doubleint_{\mathcal{D}}\Omega^{k}\varphi\cdot
Q_{0}(\nabla\Omega^{p}\varphi, \nabla\Omega^{q}\varphi)L \Omega^i\varphi\\
&\quad +\doubleint_{\mathcal{D}} \frac{1}{2r} \Lb \Omega^i \varphi \cdot L \Omega^i\varphi + \doubleint_{\mathcal{D}} \frac{1}{2r} |\nablaslash \Omega^i \varphi|^2\\
&=\int_{C_{u_{0}}}|L\Omega^i \varphi|^{2} + S_1 +S_2 +S_3 +S_4,
\end{split}
\end{equation}
where the $S_j$'s are defined in the obvious way. We also recall that, for all smooth functions $\phi$, we actually have
\begin{equation*}
 \|\Omega^i\phi\|_{L^p(S_{\ub,u})} \sim \|\nablaslash^i\phi \|_{L^p(S_{\ub,u})}.
\end{equation*}

We first consider $S_1$. In view of \eqref{null form bound} and the fact that $|\varphi|\equiv 1$, we split $S_1$ into the sum
\begin{equation*}
 S_1 \lesssim S_{11}+S_{12}+S_{13}+S_{14}+S_{15},
\end{equation*}
where
\begin{equation*}
\begin{split}
S_{11} &= \doubleint_{\mathcal{D}}|\Lb\varphi||L\Omega^i \varphi|^2, \\
S_{12} &= \doubleint_{\mathcal{D}}|L\varphi||\Lb\Omega^i\varphi||L\Omega^i \varphi|, \\
S_{13} &= \doubleint_{\mathcal{D}}|\nablaslash \varphi||\nablaslash\Omega^i\varphi||L\Omega^i \varphi|,\\
S_{14} &= \doubleint_{\mathcal{D}}|\Lb\Omega\varphi||L\Omega \varphi||L\Omega^2 \varphi|, \\
S_{15} &= \doubleint_{\mathcal{D}}|\nablaslash \Omega \varphi|^2 |L\Omega^2 \varphi|.
\end{split}
\end{equation*}
where for $S_{14}$ and $S_{15}$, $i=2$. We bound those terms one by one.

For $S_{11}$, we have
\begin{align*}
 S_{11} &\leq \int_{u_0}^{u}\|\Lb\varphi\|_{L^{\infty}} (\int_{C_{u'}} |L\Omega^i\varphi|^2 )du'\\
&\lesssim \int_{u_0}^{u}\delta^{\frac{1}{4}}M \cdot M^2 du' \\
&\lesssim \delta^{\frac{1}{4}}\cdot M^3.
\end{align*}
For the last step, we have used the fact that $u_0$ is a fixed constant.

For $S_{12}$, we have
\begin{align*}
 S_{12} &\lesssim (\doubleint_{\mathcal{D}}|L\varphi|^2|L\Omega^i\varphi|^2)^{\frac{1}{2}}(\doubleint_{\mathcal{D}}|\Lb\Omega^i \varphi|^2)^{\frac{1}{2}} \\
&\approx \|L\varphi\|_{L^{\infty}}(\int_{u_0}^{u} \|L\Omega^i\varphi\|^2_{L^2(C_{u'})} du')^{\frac{1}{2}}(\int_{0}^{\ub}\|\Lb\nablaslash^i \varphi\|_{L^2(\Cb_{\ub'})}^2 d\ub')^{\frac{1}{2}}\\
&\lesssim \delta^{-\frac{1}{2}}M \cdot M \cdot \delta M \\
&\lesssim \delta^{\frac{1}{2}} \cdot M^3.
\end{align*}

For $S_{13}$, we have
\begin{align*}
 S_{13} &\leq \int_{u_0}^{u}\|\nablaslash\varphi\|_{L^{\infty}} \|L\Omega^i\varphi\|_{L^2(C_{u'})}\|\nablaslash\Omega^i\varphi\|_{L^2(C_{u'})}du'\\
 &\lesssim \int_{u_0}^{u}\delta^{\frac{1}{4}} M \cdot M \cdot \delta^{\frac{1}{2}} M du'\\
&\lesssim \delta^{\frac{3}{4}} M^3.
\end{align*}

For $S_{14}$, we have
\begin{align*}
 S_{14} &\leq (\int_{0}^{\ub} \|\Lb\Omega \varphi\|_{L^4(\Cb_{\ub'})}^4 d\ub')^{\frac{1}{4}} (\int_{u_0}^{u} \|L \Omega \varphi\|_{L^4(C_{u'})}^4 d u')^{\frac{1}{4}}(\int_{u_0}^{u} \|L \Omega^2 \varphi\|_{L^2(C_{u'})}^2 d u')^{\frac{1}{2}}.
\end{align*}
In view of \eqref{e_8}, we have
\begin{equation*}
 \|\Lb \Omega \varphi\|_{L^4{(S_{\ub,u})}} \lesssim \delta^{\frac{1}{4}} M.
\end{equation*}
Thus,
\begin{equation*}
(\int_{0}^{\ub} \|\Lb\Omega \varphi\|_{L^4(\Cb_{\ub'})}^4 d\ub')^\frac{1}{4} \lesssim (\int_{0}^{\ub} \int_{u_0}^u \delta  M^4)^{\frac{1}{4}} \lesssim \delta^{\frac{1}{2}} M.
\end{equation*}
Similarly, we have
\begin{equation*}
 (\int_{u_0}^{u} \|L \Omega \varphi\|_{L^4(C_{u'})}^4 d u')^{\frac{1}{4}} \lesssim \delta^{-\frac{1}{4}} M.
\end{equation*}
Hence,
\begin{align*}
 S_{14} &\lesssim \delta^{\frac{1}{2}} M \cdot  \delta^{-\frac{1}{4}} M \cdot  M \lesssim \delta^{\frac{1}{4}} \cdot M^3.
\end{align*}

For $S_{15}$, we can proceed in a similar manner as for $S_{14}$. This yields
\begin{align*}
 S_{15} &\leq (\int_{u_0}^{u} \|\nablaslash \Omega \varphi\|_{L^4(C_{u'})}^4 d u')^{\frac{2}{4}}(\int_{u_0}^{u} \|L \Omega^2 \varphi\|_{L^2(C_{u'})}^2 d u')^{\frac{1}{2}} \lesssim \delta \cdot M^3.
\end{align*}

We add up those estimates for $S_{1k}$'s to derive
\begin{equation*}
 S_1 \lesssim \delta^{\frac{1}{4}} \cdot M^3.
\end{equation*}

We now consider $S_2$. Due to the symmetry of the indices $p$ and $q$ in $S_2$, we always assume $p\geq q$. Because $k \geq 1$, therefore $q =0$. We can rewrite $S_2$ as
\begin{equation*}
S_2= \sum_{ k+p=i, k\geq 1} \doubleint_{\mathcal{D}}\Omega^{k}\varphi\cdot Q_{0}(\nabla\Omega^{p}\varphi, \nabla \varphi)L \Omega^i\varphi.
\end{equation*}
In view of \eqref{null form bound}, we split $S_2$ into the sum:
\begin{equation*}
 S_2 \lesssim S_{21}+S_{22}+S_{23},
\end{equation*}
where
\begin{equation*}
\begin{split}
S_{21} &= \doubleint_{\mathcal{D}}|\Omega^k \varphi||\Lb\Omega^p \varphi||L \varphi||L\Omega^i \varphi|, \\
S_{22} &= \doubleint_{\mathcal{D}}|\Omega^k \varphi||L\Omega^p \varphi||\Lb \varphi||L\Omega^i \varphi|, \\
S_{23} &= \doubleint_{\mathcal{D}}|\Omega^k \varphi||\nablaslash\Omega^p \varphi||\nablaslash \varphi||L\Omega^i \varphi|.
\end{split}
\end{equation*}
where $i=1$ or $2$. We bound those terms one by one.

For $S_{21}$, according to the values of $(k,p)$, we have three cases: $(k,p,i) = (1,0,1)$, $(1,1,2)$ or $(2,0,2)$. Thus, we can further spit $S_{21}$ into three terms according these three cases:
\begin{equation*}
 S_{21} = S_{211}+S_{212}+S_{213},
\end{equation*}
where
\begin{equation*}
\begin{split}
S_{211} &= \doubleint_{\mathcal{D}}|\Omega \varphi||\Lb \varphi||L \varphi||L\Omega \varphi|, \\
S_{212} &= \doubleint_{\mathcal{D}}|\Omega \varphi||\Lb\Omega \varphi||L \varphi||L\Omega^2 \varphi|, \\
S_{213} &= \doubleint_{\mathcal{D}}|\Omega^2 \varphi||\Lb \varphi||L \varphi||L\Omega^2 \varphi|.
\end{split}
\end{equation*}

For $S_{211}$, we have
\begin{align*}
 S_{211} &\lesssim \int_{u_0}^{u}\|\Omega\varphi\|_{L^{\infty}}\|L\varphi\|_{L^{\infty}}\|\Lb \varphi\|_{L^2(C_{u'})}\|L\Omega\varphi\|_{L^2(C_{u'})}du'\\
&\lesssim \delta^{\frac{3}{4}}\cdot M^4.
\end{align*}

For $S_{212}$, we have
\begin{align*}
 S_{212} &\lesssim \int_{u_0}^{u}\|\Omega\varphi\|_{L^{\infty}}\|L\varphi\|_{L^{\infty}}\|\Lb\Omega\varphi\|_{L^2(C_{u'})}\|L\Omega^2\varphi\|_{L^2(C_{u'})}du'\\
&\lesssim \delta^{\frac{3}{4}} \cdot M^4.
\end{align*}

For $S_{213}$, we have
\begin{align*}
 S_{213} &\lesssim \int_{u_0}^{u}\|\Lb\varphi\|_{L^{\infty}}\|L\varphi\|_{L^{\infty}}\|\Omega^2\varphi\|_{L^2(C_{u'})}\|L\Omega^2\varphi\|_{L^2(C_{u'})}du'\\
&\lesssim  \delta^{\frac{1}{4}} \int_{u_0}^{u} du' \cdot M^4 \lesssim \delta^{\frac{1}{4}} \cdot M^4.
\end{align*}

Therefore, we have
\begin{equation*}
 S_{21} \lesssim \delta^{\frac{1}{4}} \cdot M^4.
\end{equation*}

For $S_{22}$, according to the values of $(k,p)$, we also have three cases: $(k,p,i) = (1,0,1)$, $(1,1,2)$ or $(2,0,2)$ and we further spit $S_{22}$ according these three cases:
\begin{equation*}
 S_{22} = S_{221}+S_{222}+S_{223},
\end{equation*}
where
\begin{equation*}
\begin{split}
S_{221} &= \doubleint_{\mathcal{D}}|\Omega \varphi||\Lb \varphi||L \varphi||L\Omega \varphi|, \\
S_{222} &= \doubleint_{\mathcal{D}}|\Omega \varphi||L \Omega \varphi||\Lb \varphi||L\Omega^2 \varphi|, \\
S_{223} &= \doubleint_{\mathcal{D}}|\Omega^2 \varphi||\Lb \varphi||L \varphi||L\Omega^2 \varphi|.
\end{split}
\end{equation*}
We observe that $S_{221}$ and $S_{223}$ have appeared as $S_{211}$ and $S_{213}$. So we ignore them and we only bound $S_{222}$ as follows:
\begin{align*}
 S_{222} &\lesssim \int_{u_0}^{u}\|\Omega\varphi\|_{L^{\infty}}\|\Lb \varphi\|_{L^{\infty}}\|L \Omega\varphi\|_{L^2(C_{u'})}\|L\Omega^2\varphi\|_{L^2(C_{u'})}du'\\
&\lesssim \delta^{\frac{1}{2}} \cdot M^4.
\end{align*}
Therefore,
\begin{equation*}
 S_{22} \lesssim \delta^{\frac{1}{4}} \cdot M^4.
\end{equation*}

For $S_{23}$, according to the values of $(k,p,i)$, i.e. $(k,p,i) = (1,0,1)$, $(1,1,2)$ or $(2,0,2)$, we further spit $S_{23}$ into three cases:
\begin{equation*}
 S_{23} = S_{231}+S_{232}+S_{233},
\end{equation*}
where
\begin{equation*}
\begin{split}
S_{231} &= \doubleint_{\mathcal{D}}|\Omega \varphi||\nablaslash \varphi|^2 |L\Omega \varphi|, \\
S_{232} &= \doubleint_{\mathcal{D}}|\Omega \varphi||\nablaslash\Omega \varphi||\nablaslash \varphi||L\Omega^2 \varphi|, \\
S_{233} &= \doubleint_{\mathcal{D}}|\Omega^2 \varphi||\nablaslash \varphi|^2|L\Omega^2 \varphi|.
\end{split}
\end{equation*}

For $S_{231}$, we have
\begin{align*}
 S_{231} &\lesssim \int_{u_0}^{u}\|\Omega\varphi\|_{L^{\infty}}\|\nablaslash\varphi\|_{L^{\infty}}\|\nablaslash \varphi\|_{L^2(C_{u'})}\|L\Omega\varphi\|_{L^2(C_{u'})}du'\\
&\lesssim \delta  \cdot M^4.
\end{align*}

For $S_{232}$, we have
\begin{align*}
 S_{232} &\lesssim \int_{u_0}^{u}\|\Omega\varphi\|_{L^{\infty}}\|\nablaslash\varphi\|_{L^{\infty}}\|\nablaslash \Omega\varphi\|_{L^2(C_{u'})}\|L\Omega^2\varphi\|_{L^2(C_{u'})}du'\\
&\lesssim \delta  \cdot M^4.
\end{align*}

For $S_{233}$, we have
\begin{align*}
 S_{233} &\lesssim \int_{u_0}^{u} \|\nablaslash\varphi\|^2_{L^{\infty}}\|\Omega^2\varphi\|_{L^2(C_{u'})}\|L\Omega^2\varphi\|_{L^2(C_{u'})}du'\\
&\lesssim \delta  \cdot M^4.
\end{align*}
Therefore, we have
\begin{equation*}
 S_{23} \lesssim \delta  \cdot M^4.
\end{equation*}
and
\begin{equation*}
 S_{2} \lesssim \delta^{\frac{1}{4}} \cdot M^4.
\end{equation*}

For $S_3$, we have
\begin{align*}
 S_3 &\lesssim (\int_{u_0}^{u}\frac{1}{|u'|^2}\|L\Omega^i\varphi\|^2_{L^2(C_{u'})} du')^{\frac{1}{2}}(\int_{0}^{\ub}\||u'|^i\Lb\nablaslash^i \varphi\|_{L^2(\Cb_{\ub'})}^2 d\ub')^{\frac{1}{2}}
\\
&\lesssim \delta  \cdot M^2.
\end{align*}

For $S_4$, it is similar as $S_3$. We only give the result:
\begin{align*}
 S_3 &\lesssim \delta \cdot M^2.
\end{align*}

Putting those estimates in \eqref{ES-E1-E2-E3-a}, we obtain
\begin{equation*}
\sum_{i=0}^2 \left(\int_{C_{u}}|L\Omega^i\varphi|^{2}+\int_{\underline{C}_{\underline{u}}}|\nablaslash\Omega^i\varphi|^{2}\right) \lesssim I^2 + \delta^{\frac{1}{4}} \cdot M^4.
\end{equation*}
where $I$ is the size of the initial data and $M \geq 1$. We finally obtain
\begin{equation}\label{ES-E1-E2-E3-b}
\sum_{i=0}^2 \left( \|L\Omega^i\varphi\|_{L^2(C_u)} + \|\nablaslash\Omega^i \varphi\|_{L^2(\Cb_{\ub})} \right) \lesssim I + \delta^{\frac{1}{8}}\cdot M^2.
\end{equation}

\subsection{Estimates on $E_{k}$'s and $\underline{E}_{k}$'s Part-2}

For $i=1,2$, we still consider
\begin{equation*}
 \Box \Omega^{i} \varphi = \sum_{k+p+q=i} \Omega^{k}\varphi \cdot Q_{0}(\nabla\Omega^{p}\varphi, \nabla\Omega^{q}\varphi).
\end{equation*}
and we use the scheme in Section \ref{Energy estimates scheme} where we take $\phi = \Omega^i \varphi$ ($i=0, 1, 2$) and $X = \Lb$ in \eqref{fundamental energy identity}. Therefore, we have
\begin{equation}\label{ES-Eb1-Eb2-Eb3-a}
\begin{split}
\int_{C_{u}}|\nablaslash \Omega^i \varphi|^{2}+\int_{\underline{C}_{\underline{u}}}|\Lb \Omega^i \varphi|^{2} &=\int_{C_{u_{0}}}|\nablaslash \Omega^i \varphi|^{2}+\sum_{p+q=i}\doubleint_{\mathcal{D}} \varphi Q_{0}(\nabla\Omega^{p}\phi, \nabla \Omega^{q} \phi)\Lb \Omega^i \varphi \\
&\quad + \sum_{ k+p+q=i, k\geq 1} \doubleint_{\mathcal{D}}\Omega^{k}\varphi\cdot
Q_{0}(\nabla\Omega^{p}\varphi, \nabla\Omega^{q}\varphi)\Lb \Omega^i\varphi\\
&\quad-\doubleint_{\mathcal{D}} \frac{1}{2r} \Lb \Omega^i \varphi \cdot L \Omega^i\varphi - \doubleint_{\mathcal{D}} \frac{1}{2r} |\nablaslash \Omega^i \varphi|^2\\
&=\int_{C_{u_{0}}}|\nablaslash\Omega^i \varphi|^{2} + T_1 +T_2 +T_3 +T_4,
\end{split}
\end{equation}
where the $T_j$'s are defined in the obvious way.

We first consider $T_1$. In view of \eqref{null form bound} and the fact that $|\varphi|\equiv 1$, we split $T_1$ into the sum
\begin{equation*}
 T_1 \lesssim T_{11}+T_{12}+T_{13}+T_{14}+T_{15},
\end{equation*}
where
\begin{equation*}
\begin{split}
T_{11} &= \doubleint_{\mathcal{D}}|L \varphi||\Lb\Omega^i \varphi|^2, \\
T_{12} &= \doubleint_{\mathcal{D}}|\Lb\varphi||L\Omega^i \varphi||\Lb\Omega^i\varphi|, \\
T_{13} &= \doubleint_{\mathcal{D}}|\nablaslash \varphi||\nablaslash\Omega^i\varphi||\Lb\Omega^i \varphi|,\\
T_{14} &= \doubleint_{\mathcal{D}}|\Lb\Omega\varphi||L\Omega \varphi||\Lb\Omega^2 \varphi|, \\
T_{15} &= \doubleint_{\mathcal{D}}|\nablaslash \Omega \varphi|^2 |\Lb\Omega^2 \varphi|.
\end{split}
\end{equation*}
where for $T_{14}$ and $T_{15}$, $i=2$. We bound those terms one by one.

For $T_{11}$, we have
\begin{align*}
 T_{11} &\leq \int_{0}^{\delta}\|L\varphi\|_{L^{\infty}} (\int_{\Cb_{\ub'}} |\Lb\Omega^i\varphi|^2 )d\ub'\\
 &\lesssim \delta^{\frac{3}{2}}  \cdot M^3.
\end{align*}

For $T_{12}$, we have
\begin{align*}
 T_{12} &\leq \|\Lb\varphi\|_{L^{\infty}}(\int_{u_0}^{u} \|L\Omega^i\varphi\|^2_{L^2(C_{u'})} du')^{\frac{1}{2}}(\int_{0}^{\ub}\||u'|^i\Lb\nablaslash^i \varphi\|_{L^2(\Cb_{\ub'})}^2 d\ub')^{\frac{1}{2}}\\
 &\lesssim \delta^{\frac{5}{4}} \cdot M^3.
\end{align*}

For $T_{13}$, we have
\begin{align*}
 T_{13} &\leq \int_{0}^{\delta}\|\nablaslash\varphi\|_{L^{\infty}} \|\Lb\Omega^i\varphi\|_{L^2(\Cb_{\ub'})}\|\nablaslash\Omega^i\varphi\|_{L^2(\Cb_{\ub'})}d\ub'\\
 &\lesssim \delta^{\frac{7}{4}} \cdot M^3.
\end{align*}

For $T_{14}$, we have
\begin{align*}
 T_{14} &\leq (\int_{0}^{\ub} \|\Lb\Omega \varphi\|_{L^4(\Cb_{\ub'})}^4 d\ub')^{\frac{1}{4}} (\int_{u_0}^{u} \|L \Omega \varphi\|_{L^4(C_{u'})}^4 d u')^{\frac{1}{4}}(\int_{0}^{\ub} \|\Lb \Omega^2 \varphi\|_{L^2(\Cb_{\ub'})}^2 d \ub')^{\frac{1}{2}}.
\end{align*}
As we have done for $S_{14}$ in last subsection, we have
\begin{equation*}
(\int_{0}^{\ub} \|\Lb\Omega \varphi\|_{L^4(\Cb_{\ub'})}^4 d\ub')^\frac{1}{4}  \lesssim \delta^{\frac{1}{2}} M,
\end{equation*}
and
\begin{equation*}
 (\int_{u_0}^{u} \|L \Omega \varphi\|_{L^4(C_{u'})}^4 d u')^{\frac{1}{4}} \lesssim \delta^{-\frac{1}{4}} M.
\end{equation*}
Thus,
\begin{align*}
 T_{14} &\lesssim \delta^{\frac{5}{4}} \cdot M^3.
\end{align*}

For $T_{15}$, similarly, we have
\begin{align*}
 T_{15} &\leq (\int_{u_0}^{u} \|\nablaslash \Omega \varphi\|_{L^4(C_{u'})}^4 d u')^{\frac{2}{4}}(\int_{0}^{\ub} \|\Lb \Omega^2 \varphi\|_{L^2(\Cb_{\ub'})}^2 d \ub')^{\frac{1}{2}} \lesssim \delta^{\frac{5}{4}}  \cdot M^3.
\end{align*}

We turn to $T_2$. Due to symmetry considerations, we always assume $p\geq q$. Because $k \geq 1$, then $q =0$. We rewrite $T_2$ as
\begin{equation*}
T_2= \sum_{ k+p=i, k\geq 1} \doubleint_{\mathcal{D}}\Omega^{k}\varphi\cdot Q_{0}(\nabla\Omega^{p}\varphi, \nabla \varphi) \cdot \Lb \Omega^i\varphi.
\end{equation*}
In view of \eqref{null form bound}, we split $T_2$ into the sum:
\begin{equation*}
 T_2 \lesssim T_{21}+T_{22}+T_{23},
\end{equation*}
where
\begin{equation*}
\begin{split}
T_{21} &= \doubleint_{\mathcal{D}}|\Omega^k \varphi||\Lb\Omega^p \varphi||L \varphi||\Lb\Omega^i \varphi|, \\
T_{22} &= \doubleint_{\mathcal{D}}|\Omega^k \varphi||L\Omega^p \varphi||\Lb \varphi||\Lb\Omega^i \varphi|, \\
T_{23} &= \doubleint_{\mathcal{D}}|\Omega^k \varphi||\nablaslash\Omega^p \varphi||\nablaslash \varphi||\Lb\Omega^i \varphi|.
\end{split}
\end{equation*}
where $i=1$ or $2$ and $k+p=i$. We bound those terms one by one.

For $T_{21}$, according to the values of $(k,p,i)$, we have three cases: $(k,p,i) = (1,0,1)$, $(1,1,2)$ or $(2,0,2)$. Thus, we can further spit $T_{21}$ into three terms according these three cases:
\begin{equation*}
 T_{21} = T_{211}+T_{212}+T_{213},
\end{equation*}
where
\begin{equation*}
\begin{split}
T_{211} &= \doubleint_{\mathcal{D}}|\Omega \varphi||\Lb \varphi||L \varphi||\Lb\Omega \varphi|, \\
T_{212} &= \doubleint_{\mathcal{D}}|\Omega \varphi||\Lb\Omega \varphi||L \varphi||\Lb\Omega^2 \varphi|, \\
T_{213} &= \doubleint_{\mathcal{D}}|\Omega^2 \varphi||\Lb \varphi||L \varphi||\Lb\Omega^2 \varphi|.
\end{split}
\end{equation*}

For $T_{211}$, we have
\begin{align*}
 T_{211} &\lesssim \int_0^{\ub}\|\Omega\varphi\|_{L^{\infty}}\|L\varphi\|_{L^{\infty}}\|\Lb \varphi\|_{L^2(\Cb_{\ub'})}\|\Lb\Omega\varphi\|_{L^2(\Cb_{\ub'})}d\ub'\\
&\lesssim \delta^{\frac{7}{4}} \cdot M^4.
\end{align*}

For $T_{212}$, we have
\begin{align*}
 T_{212} &\lesssim \int_0^{\ub}\|\Omega\varphi\|_{L^{\infty}}\|L\varphi\|_{L^{\infty}} \|\Lb\Omega\varphi\|_{L^2(\Cb_{\ub'})}\|\Lb\Omega^2\varphi\|_{L^2(\Cb_{\ub'})}d\ub'\\
&\lesssim \delta^{\frac{7}{4}} \cdot M^4.
\end{align*}

For $T_{213}$, we have
\begin{align*}
 T_{213} &\lesssim \int_0^{\ub}\|\Lb\varphi\|_{L^{\infty}}\|L\varphi\|_{L^{\infty}}\|\Omega^2\varphi\|_{L^2(\Cb_{\ub'})}\|\Lb\Omega^2\varphi\|_{L^2(\Cb_{\ub'})}d \ub'\\
&\lesssim \delta^{\frac{5}{4}} \cdot M^4.
\end{align*}

Therefore, we have
\begin{equation*}
 T_{21} \lesssim \delta^{\frac{5}{4}} \cdot M^4.
\end{equation*}

For $T_{22}$, according to the values of $(k,p)$, we also have three cases: $(k,p,i) = (1,0,1)$, $(1,1,2)$ or $(2,0,2)$ and we further spit $T_{22}$ according these three cases:
\begin{equation*}
 T_{22} = T_{221}+T_{222}+T_{223},
\end{equation*}
where
\begin{equation*}
\begin{split}
T_{221} &= \doubleint_{\mathcal{D}}|\Omega \varphi||\Lb \varphi||L \varphi||\Lb\Omega \varphi|, \\
T_{222} &= \doubleint_{\mathcal{D}}|\Omega \varphi||L \Omega \varphi||\Lb \varphi||\Lb\Omega^2 \varphi|, \\
T_{223} &= \doubleint_{\mathcal{D}}|\Omega^2 \varphi||\Lb \varphi||L \varphi||\Lb\Omega^2 \varphi|.
\end{split}
\end{equation*}
Since $T_{221}$ and $T_{223}$ have appeared as $T_{211}$ and $T_{213}$, we ignore them and we only bound $T_{222}$ as follows:
\begin{align*}
 T_{222} &\lesssim \int_{0}^{\ub}\|\Omega\varphi\|_{L^{\infty}}\|\Lb \varphi\|_{L^{\infty}}\|L \Omega\varphi\|_{L^2(\Cb_{\ub'})}\|\Lb \Omega^2\varphi\|_{L^2(\Cb_{\ub'})}d \ub'\\
&\lesssim \delta  \cdot M^3 \cdot \int_{0}^{\ub}\|L \Omega\varphi\|_{L^2(\Cb_{\ub'})}d \ub'\\
&\lesssim \delta^{\frac{3}{2}}  \cdot M^3 \cdot \left(\int_{0}^{\ub}\|L \Omega\varphi\|^2_{L^2(\Cb_{\ub'})}d \ub'\right)^{\frac{1}{2}}\\
&\lesssim \delta^{2}  \cdot M^4.
\end{align*}
Therefore,
\begin{equation*}
 T_{22} \lesssim  \delta^{\frac{5}{4}}  \cdot M^4.
\end{equation*}

For $T_{23}$, according to the values of $(k,p,i)$, i.e. $(k,p,i) = (1,0,1)$, $(1,1,2)$ or $(2,0,2)$, we further spit $T_{23}$ into three cases:
\begin{equation*}
 T_{23} = T_{231}+T_{232}+T_{233},
\end{equation*}
where
\begin{equation*}
\begin{split}
T_{231} &= \doubleint_{\mathcal{D}}|\Omega \varphi||\nablaslash \varphi|^2 |\Lb\Omega \varphi|, \\
T_{232} &= \doubleint_{\mathcal{D}}|\Omega \varphi||\nablaslash\Omega \varphi||\nablaslash \varphi||\Lb\Omega^2 \varphi|, \\
T_{233} &= \doubleint_{\mathcal{D}}|\Omega^2 \varphi||\nablaslash \varphi|^2|\Lb\Omega^2 \varphi|.
\end{split}
\end{equation*}

For $T_{231}$, we have
\begin{align*}
 T_{231} &\lesssim \int_{0}^{\ub}\|\Omega\varphi\|_{L^{\infty}}\|\nablaslash\varphi\|_{L^{\infty}}\|\nablaslash \varphi\|_{L^2(\Cb_{\ub'})}\|\Lb\Omega\varphi\|_{L^2(\Cb_{\ub'})} d \ub'\\
&\lesssim \delta^2  \cdot M^4.
\end{align*}

For $T_{232}$, we have
\begin{align*}
 T_{232} &\lesssim \int_{0}^{\ub}\|\Omega\varphi\|_{L^{\infty}}\|\nablaslash\varphi\|_{L^{\infty}}\|\nablaslash \Omega\varphi\|_{L^2(\Cb_{\ub'})}\|\Lb\Omega^2\varphi\|_{L^2(\Cb_{\ub'})}d\ub'\\
&\lesssim \delta^2  \cdot M^4.
\end{align*}

For $T_{233}$, we have
\begin{align*}
 T_{233} &\lesssim \int_{0}^{\ub} \|\nablaslash\varphi\|^2_{L^{\infty}}\|\Omega^2\varphi\|_{L^2(\Cb_{\ub'})}\|\Lb\Omega^2\varphi\|_{L^2(C_{u'})}d\ub'\\
&\lesssim \delta^2 \cdot M^4.
\end{align*}
Therefore, we have
\begin{equation*}
 T_{23} \lesssim \delta^2 \cdot M^4.
\end{equation*}
and
\begin{equation*}
 T_{2} \lesssim \delta^{\frac{5}{4}} \cdot M^4.
\end{equation*}

For $T_3$, the derivation of the estimates is different from all the previous terms. In fact, we have
\begin{align*}
T_3 &\lesssim \doubleint_{\mathcal{D}}\frac{\delta}{r^2}|L\Omega^i\varphi|^2+ \frac{1}{\delta}|\Lb\Omega^i\varphi|^2\\
&= \int_{u_0}^{u} \delta  \|L\Omega^i\varphi\|^2_{L^{2}(C_{u'})} du' + \frac{1}{\delta}\int_{0}^{\ub}\|\Lb\Omega^i\varphi\|^2_{L^2(\Cb_{\ub'})} d\ub'.
\end{align*}
At this point, we can use the conclusion of \eqref{ES-E1-E2-E3-b} which leads to
\begin{align*}
T_3 &\lesssim \int_{u_0}^{u} \delta  (I^2 + \delta^{\frac{1}{8}}M^{4}) du' + \frac{1}{\delta}\int_{0}^{\ub}\|\Lb\Omega^i\varphi\|^2_{L^2(\Cb_{\ub'})} d\ub'\\
&= \delta I^2 + \delta^{\frac{9}{8}} \cdot M^{4}   + \frac{1}{\delta}\int_{0}^{\ub}\|\Lb\Omega^i\varphi\|^2_{L^2(\Cb_{\ub'})} d\ub'.
\end{align*}

For $T_4$, again thanks to \eqref{ES-E1-E2-E3-b}, we have
\begin{align*}
T_4 &\lesssim \int_{0}^{\ub} \int_{\Cb_{\ub'}}|\nablaslash\Omega^i\varphi|^2 d\ub'\\
&\lesssim \delta I^2 + \delta^{\frac{9}{8}} \cdot M^4.
\end{align*}

Now we add all \eqref{ES-Eb1-Eb2-Eb3-a} for $i=0,1,2$, in view of the above estimates, if $\delta$ is sufficiently small, we now have
\begin{equation*}
\sum_{i=0}^2 \left(\int_{C_{u}}|\nablaslash\Omega^i\varphi|^{2}+\int_{\underline{C}_{\underline{u}}}|\Lb\Omega^i\varphi|^{2}\right)\lesssim \delta I^2 + \delta^{\frac{9}{8}} \cdot M^4  + \frac{1}{\delta}\sum_{i=0}^2\int_{0}^{\ub}\|\Lb\Omega^i\varphi\|^2_{L^2(\Cb_{\ub'})} d\ub'.
 \end{equation*}
Since the $\|\Lb\Omega^i\varphi\|^2_{L^2(\Cb_{\ub})}$'s also appear on the left hand side, a standard use of Gronwall's inequality removes the integral on the right hand side. This yields
\begin{equation}\label{ES-Eb1-Eb2-Eb3-b}
\sum_{i=0}^2 \left(\|\nablaslash\Omega^i\varphi\|_{L^2(C_u)} + \|\Lb\Omega^i \varphi\|_{L^2(\Cb_{\ub})} \right)\lesssim \delta^{\frac{1}{2}} I + \delta^{\frac{9}{16}} \cdot M^{2}.
\end{equation}

Putting \eqref{ES-E1-E2-E3-b} and \eqref{ES-Eb1-Eb2-Eb3-b} together, we have obtained the energy estimates for $E_{k}$'s and $\underline{E}_{k}$'s:
\begin{equation}\label{ES E-Eb up to three derivative}
\sum_{i=1}^3 \left (E_i(u,\ub) + \Eb_i(u, \ub)\right) \lesssim I +  \delta^{\frac{1}{16}} \cdot M^{2}.
\end{equation}

\subsection{Estimates on $F_2(u,\ub)$ and $\Fb_2(u, \ub)$}
We first consider the bound of $\|\Lb^2 \varphi\|_{L^2(\Cb_{\ub})}$. We commute $\Lb$ with \eqref{Main Equation}, in view of \eqref{commute Lb Omega n with main equation}, we obtain
\begin{equation*}
\begin{split}
 \Box \Lb \varphi &= \Lb \varphi \cdot Q_{0}(\nabla\varphi, \nabla\varphi)+ \varphi \cdot Q_{0}(\nabla \Lb\varphi, \nabla\varphi) +\varphi \cdot \frac{2}{r} \left( Q_{0}(\nabla \varphi, \nabla\varphi)+L\varphi \Lb\varphi \right)\\
& \qquad\qquad\qquad + \frac{1}{2r^2}(\Lb  \varphi -L\varphi)-\frac{2}{r}\laplacianslash \varphi.
\end{split}
\end{equation*}

We use \eqref{fundamental energy identity} for the above equation where we take $\phi = \Lb\varphi$ and $X =\Lb$, therefore,\footnote{\,\, Once again, we ignore the signs and numerical constants.}
\begin{equation}\label{ES-2-3-a}
\begin{split}
&\quad\int_{C_{u}}|\nablaslash \Lb\varphi|^{2}+\int_{\underline{C}_{\underline{u}}}|\Lb^2\varphi|^{2} =\int_{C_{u_{0}}}|\nablaslash \Lb\varphi|^{2}  +\doubleint_{\mathcal{D}}\varphi Q_{0}(\nabla \Lb \varphi, \nabla \varphi)
\Lb^2\varphi \\
&\quad +\doubleint_{\mathcal{D}}\varphi \cdot\frac{2}{r}\left(Q_{0}(\nabla \varphi, \nabla \varphi)+L\varphi\Lb \varphi
\right)\Lb^2\varphi + \doubleint_{\mathcal{D}}\Lb\varphi Q_{0}(\nabla \varphi, \nabla \varphi)\Lb^2\varphi\\
&\quad +\doubleint_{\mathcal{D}}\frac{1}{r^2}(\Lb \varphi -L \varphi)\Lb^2 \varphi + \doubleint_{\mathcal{D}}\frac{1}{r}\laplacianslash \varphi \Lb^2 \varphi+ \doubleint_{\mathcal{D}} \frac{1}{2r}L\Lb\varphi \Lb^2 \varphi + \doubleint_{\mathcal{D}} \frac{1}{2r}|\nablaslash\Lb\varphi|^2\\
&= \int_{C_{u_{0}}}|\nablaslash \Lb\varphi|^{2} + S_1+S_2+\cdots+S_7,
\end{split}
\end{equation}
where the error terms $S_j$'s are defined in the obvious way. We now bound those error terms one by one.

For $S_1$, we have
\begin{equation*}
\begin{split}
 S_1 &\lesssim \doubleint_{\mathcal{D}}|L \varphi||\Lb^2 \varphi|^2 + \doubleint_{\mathcal{D}}|\Lb \varphi||L\Lb \varphi||\Lb^2 \varphi| +\doubleint_{\mathcal{D}}|\nablaslash \varphi||\nablaslash\Lb \varphi||\Lb^2 \varphi|\\
&=S_{11}+S_{12}+S_{13}.
\end{split}
\end{equation*}

For $S_{11}$, we have
\begin{align*}
 S_{11}  &\lesssim \|L\varphi\|_{L^\infty} \int_{0}^{\ub} (\int_{\Cb_{\ub'}}|\Lb^2 \varphi|^2)d\ub'\\
&\lesssim \delta^{\frac{1}{2}} \cdot M^3.
\end{align*}

For $S_{12}$, we need control for $L \Lb \varphi$. According to \eqref{Main Equation in null frame}, we have
\begin{align*}
 \|L\Lb\varphi\|_{L^2(C_u)} &\lesssim \|\laplacianslash \varphi\|_{L^2(C_u)}+\|\frac{1}{r} L\varphi\|_{L^2(C_u)}+\|\frac{1}{r}\Lb\varphi\|_{L^2(C_u)}\\
&\quad 	+\|\Lb\varphi  L\varphi\|_{L^2(C_u)}+\||\nablaslash \varphi|^2\|_{L^2(C_u)}.
\end{align*}
For quadratic terms, we bound one term in $L^\infty$, thus, we have
\begin{equation}\label{estimates on L Lb phi}
 \|L\Lb\varphi\|_{L^2(C_u)} \lesssim M.
\end{equation}
Hence,
\begin{align*}
 S_{12}  &\lesssim \|\Lb\varphi\|_{L^{\infty}} (\int_{u_0}^{u}\|L\Lb\varphi\|^2_{L^2(C_{u'})}du')^{\frac{1}{2}}(\int_{0}^{\ub}\|\Lb^2\varphi\|^2_{L^2(\Cb_{\ub'})} d \ub')^{\frac{1}{2}}\\
&\lesssim \delta^{\frac{3}{4}} \cdot M^3.
\end{align*}

For $S_{13}$, we have
\begin{align*}
  S_{13} & \lesssim \|\nablaslash \varphi\|_{L^{\infty}} (\int_{u_0}^{u}\|\nablaslash \Lb\varphi\|^2_{L^2(C_{u'})}du')^{\frac{1}{2}}(\int_{0}^{\ub}\|\Lb^2\varphi\|^2_{L^2(\Cb_{\ub'})} d \ub')^{\frac{1}{2}}\\
&\lesssim \delta^{\frac{7}{4}} \cdot M^3.
\end{align*}

Thus, we have
\begin{align*}
  S_{1} &\lesssim \delta^{\frac{1}{2}} \cdot M^3.
\end{align*}

For $S_2$, we have
\begin{align*}
 S_2 &\lesssim \doubleint_{\mathcal{D}} (|\nablaslash \varphi|^2+|L \varphi||\Lb\varphi|) |\Lb^2 \varphi|\\
&\lesssim  (\|\nablaslash \varphi\|_{L^{\infty}}+\|\Lb \varphi\|_{L^{\infty}}) (\|\nablaslash\varphi\|_{L^2(\mathcal{D})}+ \|L\varphi\|_{L^2(\mathcal{D})})\|\Lb^2\varphi\|_{L^2(\mathcal{D})}\\
& \lesssim \delta \cdot M^3.
\end{align*}

For $S_3$, we have	
\begin{align*}
 S_3 &\lesssim\|\Lb \varphi\|_{L^\infty} \doubleint_{\mathcal{D}}(|\nablaslash \varphi|^2+|L \varphi||\Lb\varphi|) |\Lb^2 \varphi|\\
&\lesssim \delta^{\frac{1}{4}} \cdot M (\|\nablaslash \varphi\|_{L^{\infty}}+\|\Lb \varphi\|_{L^{\infty}}) (\|\nablaslash\varphi\|_{L^2(\mathcal{D})}+ \|L\varphi\|_{L^2(\mathcal{D})})\|\Lb^2\varphi\|_{L^2(\mathcal{D})}\\
& \lesssim \delta^{\frac{5}{4}}  \cdot M^4.
\end{align*}

For $S_4$, $S_5$, $S_6$ and $S_7$, we use H\"{o}lder's inequality to bound each factor in the quadratic expressions in $L^2$. As an example, we bound $S_4$ as follows:
\begin{align*}
 S_4 &\lesssim (\|L\varphi\|_{L^2(\mathcal{D})}+ \|\Lb\varphi\|_{L^2(\mathcal{D})})\|\Lb^2\varphi\|_{L^2(\mathcal{D})}\\
& \lesssim \delta^{\frac{1}{2}} \cdot M^2.
\end{align*}
Similarly, we have
\begin{equation*}
 S_4 +S_5 +S_6 +S_7 \lesssim \delta^{\frac{1}{2}} \cdot M^2.
\end{equation*}

Putting all those estimates back to \eqref{ES-2-3-a},  we obtain
\begin{equation*}
\int_{\underline{C}_{\underline{u}}}|\Lb^2\varphi|^{2} \lesssim \delta^{2} I^2 + \delta^{\frac{1}{2}} M^3+\delta^{\frac{3}{4}} M^4.
\end{equation*}
where $I$ is the size of the initial data. This leads to our desired estimates:
\begin{equation}\label{ES-2-3-b}
\Fb_2(\ub)\lesssim \delta I + 	\delta^{\frac{1}{4}} M^{\frac{3}{2}}.
\end{equation}
\\

We now consider the bound of $\|L^2 \varphi\|_{L^2(C_u)}$. In view of \eqref{commute L Omega n with main equation}, we commute $L$ with \eqref{Main Equation} to derive
\begin{equation*}
\begin{split}
 \Box L \varphi &= L\varphi \cdot
Q_{0}(\nabla \varphi, \nabla \varphi)+\varphi \cdot Q_{0}(\nabla L\varphi, \nabla\varphi)  + \varphi \cdot \frac{2}{r}(Q_{0}(\nabla \varphi, \nabla \varphi)+L\varphi \Lb\varphi)\\
&\qquad \qquad \qquad - \frac{1}{2r^2}(\Lb \varphi -L  \varphi)+\frac{2}{r}\laplacianslash  \varphi.
\end{split}
\end{equation*}

We use \eqref{fundamental energy identity} for the above equation where we take $\phi = L\varphi$ and $X = L$, therefore,\footnote{\,\, Once again, we ignore the signs and numerical constants.}
\begin{equation}\label{ES-2-4-a}
\begin{split}
&\quad\int_{C_{u}}|L^2 	 \varphi|^{2}+\int_{\underline{C}_{\underline{u}}}|\nablaslash L \varphi|^{2} =\int_{C_{u_{0}}}|L^2\varphi|^{2}  +\doubleint_{\mathcal{D}}\varphi Q_{0}(\nabla L \varphi, \nabla \varphi)
L^2\varphi \\
&\quad +\doubleint_{\mathcal{D}}\varphi \cdot\frac{2}{r}\left(Q_{0}(\nabla \varphi, \nabla \varphi)+L\varphi\Lb \varphi
\right)L^2\varphi + \doubleint_{\mathcal{D}}L\varphi Q_{0}(\nabla \varphi, \nabla \varphi)L^2\varphi\\
&\quad +\doubleint_{\mathcal{D}}\frac{1}{r^2}(\Lb \varphi -L \varphi)L^2 \varphi + \doubleint_{\mathcal{D}}\frac{1}{r}\laplacianslash \varphi L^2 \varphi+ \doubleint_{\mathcal{D}} \frac{1}{2r}L\Lb\varphi L^2 \varphi + \doubleint_{\mathcal{D}} \frac{1}{2r}|\nablaslash L\varphi|^2\\
&= \int_{C_{u_{0}}}|L^2\varphi|^{2}   + T_1+T_2+\cdots+T_7,
\end{split}
\end{equation}
where the error terms $T_j$'s are defined in the obvious way. We now bound those error terms one by one.

For $T_1$, we have
\begin{equation*}
\begin{split}
 T_1 &\lesssim \doubleint_{\mathcal{D}}|\Lb \varphi||L^2 \varphi|^2 + \doubleint_{\mathcal{D}}|L \varphi||\Lb L \varphi||L^2 \varphi| +\doubleint_{\mathcal{D}}|\nablaslash \varphi||\nablaslash L \varphi||L^2 \varphi|\\
&=T_{11}+T_{12}+T_{13}.
\end{split}
\end{equation*}

For $T_{11}$, we have
\begin{align*}
 T_{11}  &\lesssim \|\Lb\varphi\|_{L^\infty} \int_{u_0}^{u} (\int_{C_{u'}}|L^2 \varphi|^2) d u'\\
&\lesssim \delta^{-\frac{7}{4}} \cdot M^3.
\end{align*}

For $T_{12}$, we have
\begin{align*}
 T_{12}  &\lesssim \|L \varphi\|_{L^{\infty}} (\int_{u_0}^{u}\|L\Lb\varphi\|^2_{L^2(C_{u'})}du')^{\frac{1}{2}}(\int_{u_0}^{u}\|L^2 \varphi\|^2_{L^2(C_{u'})}du')^{\frac{1}{2}}\\
&\lesssim \delta^{-\frac{3}{2}} \cdot M^3.
\end{align*}

For $T_{13}$, we have
\begin{align*}
  T_{13} & \lesssim \|\nablaslash \varphi\|_{L^{\infty}} (\int_{u_0}^{u}\|\nablaslash \Lb\varphi\|^2_{L^2(C_{u'})}du')^{\frac{1}{2}}(\int_{u_0}^{u}\|L^2 \varphi\|^2_{L^2(C_{u'})}du')^{\frac{1}{2}}\\
&\lesssim \delta^{-\frac{3}{2}} \cdot M^3.
\end{align*}

Thus, we have
\begin{align*}
  T_{1} &\lesssim \delta^{-\frac{7}{4}} \cdot M^3.
\end{align*}

For $T_2$, we have
\begin{align*}
 T_2 &\lesssim \doubleint_{\mathcal{D}} (|\nablaslash \varphi|^2+|L \varphi||\Lb\varphi|) |L^2 \varphi|\\
&\lesssim  (\|\nablaslash \varphi\|_{L^{\infty}}+\|\Lb \varphi\|_{L^{\infty}}) (\|\nablaslash\varphi\|_{L^2(\mathcal{D})}+ \|L\varphi\|_{L^2(\mathcal{D})})\|\Lb^2\varphi\|_{L^2(\mathcal{D})}\\
&\lesssim \delta^{-\frac{5}{4}} \cdot M^3.
\end{align*}

For $T_3$, similarly, we have	
\begin{align*}
 T_3 &\lesssim  \delta^{-\frac{3}{4}} \cdot M^4.
\end{align*}

For $T_4$, $T_5$, $T_6$ and $T_7$, we bound each factor in the quadratic expressions in $L^2$. To illustrate, we bound $S_4$ as follows:
\begin{align*}
 T_4 &\lesssim  (\|L\varphi\|_{L^2(\mathcal{D})}+ \|\Lb\varphi\|_{L^2(\mathcal{D})})\|L^2\varphi\|_{L^2(\mathcal{D})}\\
& \lesssim \delta^{-1} \cdot M^2.
\end{align*}
Similarly, we have
\begin{equation*}
 T_4 +T_5 + T_6 +T_7 \lesssim \delta^{-1} \cdot M^2.
\end{equation*}

Back to \eqref{ES-2-4-a}, we have
\begin{equation*}
\int_{C_{u}}|L^2 \varphi|^{2} \lesssim \delta^{-2}I^2  + \delta^{-\frac{7}{4}} M^3,
\end{equation*}
thus,
\begin{equation}\label{ES-2-4-b}
F_2(u,\ub)= \delta\| L^{2}\varphi\|_{L^{2}(C_{u})} \lesssim I  + \delta^{\frac{1}{8}} M^\frac{3}{2}.
\end{equation}

We summarize the estimates in the subsection, i.e. \eqref{ES-2-4-b} and \eqref{ES-2-3-b} as follows,
\begin{equation}\label{ES two derivative F2 Fb2}
\begin{split}
F_2(u,\ub) &\lesssim I  + \delta^{\frac{1}{16}}\cdot M^\frac{3}{2}, \\
\Fb_2(u, \ub) &\lesssim \delta I + \delta^{\frac{1}{8}}\cdot M^\frac{3}{2}.
\end{split}
\end{equation}
where $\delta$ is sufficiently small.

\subsection{Estimates on $F_{3}(u,\ub)$}

In view of \eqref{commute L Omega n with main equation}, we commute $L$ and $\Omega$ with \eqref{Main Equation} to derive
\begin{equation*}
\begin{split}
 \Box L \Omega \varphi = & L\varphi \cdot Q_0(\nabla \Omega \varphi, \nabla \varphi)+L\Omega \varphi \cdot Q_0(\nabla \varphi, \nabla \varphi)\\
&+\varphi\cdot Q_0(\nabla L\Omega \varphi, \nabla \varphi) + \varphi \cdot Q_0(\nabla \Omega \varphi, \nabla L \varphi) + \Omega \varphi \cdot Q_0(\nabla L\varphi, \nabla \varphi)\\
&+ \varphi\cdot \frac{1}{r}\cdot \left( 2Q_0(\nabla \Omega \varphi, \nabla \varphi) + L\Omega \varphi \cdot \Lb\varphi + L\varphi \cdot \Lb \Omega \varphi\right)\\
&+ \Omega \varphi\cdot \frac{2}{r}\cdot \left( Q_0(\nabla \varphi, \nabla \varphi) + L\varphi \cdot \Lb \varphi\right)- \frac{1}{2r^2}(\Lb \Omega \varphi -L \Omega \varphi)+\frac{2}{r}\laplacianslash  \varphi.
\end{split}
\end{equation*}
Applying \eqref{fundamental energy identity} to the above equation with $\phi = L\Omega\varphi,$ and $X = L$, we obtain
\begin{equation}\label{ES-F3-F4}
\begin{split}
\int_{C_{u}}&|L^2 \Omega \varphi|^{2}+\int_{\underline{C}_{\underline{u}}}|\nablaslash L \Omega\varphi|^{2} =\int_{C_{u_{0}}}|L^2 \Omega \varphi|^{2} + \doubleint_{\mathcal{D}}L\varphi \cdot Q_0(\nabla \Omega \varphi, \nabla \varphi) \cdot L^2\Omega\varphi\\
&+ \doubleint_{\mathcal{D}}L\Omega \varphi \cdot Q_0(\nabla \varphi, \nabla \varphi) \cdot L^2\Omega\varphi + \doubleint_{\mathcal{D}}\varphi\cdot Q_0(\nabla L\Omega \varphi, \nabla \varphi)\cdot L^2\Omega\varphi\\
&+ \doubleint_{\mathcal{D}}\varphi \cdot Q_0(\nabla \Omega \varphi, \nabla L \varphi)  \cdot L^2\Omega\varphi + \doubleint_{\mathcal{D}}\Omega \varphi \cdot Q_0(\nabla L\varphi, \nabla \varphi)\cdot L^2\Omega\varphi\\
&+ \doubleint_{\mathcal{D}}\varphi\cdot \frac{1}{r}\cdot \left( 2Q_0(\nabla \Omega \varphi, \nabla \varphi) + L\Omega \varphi \cdot \Lb\varphi + L\varphi \cdot \Lb \Omega \varphi\right) \cdot L^2\Omega\varphi\\
&+ \doubleint_{\mathcal{D}}\Omega \varphi\cdot \frac{2}{r}\cdot \left( Q_0(\nabla \varphi, \nabla \varphi) + L\varphi \cdot \Lb \varphi\right)\cdot L^2\Omega\varphi\\
&+\doubleint_{\mathcal{D}}  \frac{1}{r^2}(\Lb \Omega \varphi -L \Omega \varphi)\cdot L^2\Omega\varphi+\doubleint_{\mathcal{D}}\frac{2}{r}\laplacianslash  \varphi\cdot L^2\Omega\varphi\\
&+\doubleint_{\mathcal{D}}\frac{1}{r}\Lb L \Omega \varphi \cdot L^2\Omega\varphi+\doubleint_{\mathcal{D}}\frac{1}{r} |\nablaslash L\Omega\varphi|^2.
\end{split}
\end{equation}
We rewrite the right hand side of the above equation as
\begin{equation*}
 \int_{C_{u_{0}}}|L^2 \Omega \varphi|^{2} + S_1 +S_2 + \cdots+ S_{11},
\end{equation*}
where the error terms $S_i$'s are defined in the obvious way. Once again, we need to control the error terms one by one.

For $S_1$, we have
\begin{align*}
 S_1 &\lesssim \doubleint_{\mathcal{D}}|L\varphi|^2 |\Lb \Omega \varphi| |L^2\Omega\varphi|+\doubleint_{\mathcal{D}}|L\varphi| |\Lb \varphi||L \Omega \varphi| |L^2\Omega\varphi| + \doubleint_{\mathcal{D}}|L\varphi| |\nablaslash \varphi| |\nablaslash \Omega \varphi| |L^2\Omega\varphi|\\
&=S_{11} +S_{12} +S_{13}.
\end{align*}

For $S_{11}$, we have
\begin{align*}
 S_{11}  &\lesssim \|L\varphi\|^2_{L^\infty} \int_{u_0}^{u} \|\Lb\Omega \varphi\|_{L^2(C_{u'})} \|L^2\Omega\varphi\|_{L^2(C_{u'})} d u'\\
&\lesssim \delta^{-1}\cdot M^4.
\end{align*}

For $S_{12}$ and $S_{13}$, similarly, we have
\begin{align*}
 S_{12}+S_{13} \lesssim \delta^{-1} \cdot M^4.
\end{align*}

Thus, we have
\begin{align*}
 S_{1} \lesssim \delta^{-1}  \cdot M^4.
\end{align*}

For $S_2$, we have
\begin{align*}
 S_2 &\lesssim \doubleint_{\mathcal{D}}|L\varphi||\Lb \varphi| |L \Omega \varphi| |L^2\Omega\varphi|+\doubleint_{\mathcal{D}}|\nablaslash\varphi|^2 |L \Omega \varphi| |L^2\Omega\varphi|.
\end{align*}
Each term in the above expression can be estimated exactly in a similar manner as we have done for $S_{11}$, therefore,
\begin{align*}
 S_{2} \lesssim \delta^{-1} \cdot M^4.
\end{align*}

For $S_3$, we have
\begin{align*}
 S_3 &\lesssim  \doubleint_{\mathcal{D}}|\Lb\varphi| |L^2\Omega\varphi|^2 + \doubleint_{\mathcal{D}}|L\varphi||\Lb L\Omega\varphi| |L^2\Omega\varphi| + \doubleint_{\mathcal{D}}|\nablaslash\varphi||\nablaslash L\Omega\varphi| |L^2\Omega\varphi|\\
&\lesssim S_{31} +S_{32}+S_{33}.
\end{align*}

The estimates for $S_{31}$ and $S_{33}$ are straightforward, we simply use $L^\infty$ bound on the first order term. This yields
\begin{align*}
 S_{31} +S_{33} \lesssim \delta^{-\frac{3}{2}} \cdot M^3.
\end{align*}

For $S_{32}$, we need the bound on $\|\Lb L\Omega\varphi\|_{L^2(C_u)}$. Since the derivation is exactly similar to that of \eqref{estimates on L Lb phi}, we only state the conclusion:
\begin{equation}\label{estimates on L Lb Omega phi}
 \|\Lb L\Omega\varphi\|_{L^2(C_u)} \lesssim  M.
\end{equation}
Then, we can bound $S_{32}$ in a similar manner as $S_{31}$ or $S_{32}$. This yields
\begin{align*}
 S_{32} \lesssim \delta^{-\frac{3}{2}}  \cdot M^3,
\end{align*}
and
\begin{align*}
 S_{3} \lesssim \delta^{-\frac{3}{2}} \cdot M^3,
\end{align*}

For $S_{4}$, we have
\begin{align*}
 S_{4} &\lesssim  \doubleint_{\mathcal{D}}|\Lb\Omega\varphi||L^2 \varphi||L^2\Omega\varphi| + \doubleint_{\mathcal{D}}|L\Omega\varphi| |\Lb L \varphi||L^2\Omega\varphi| + \doubleint_{\mathcal{D}}|\nablaslash \Omega \varphi||\nablaslash L \varphi| |L^2\Omega\varphi|\\
&=S_{41} +S_{42} +S_{43}.
\end{align*}

For $S_{41}$, we have
\begin{align*}
 S_{41}  &\leq \int_0^{\ub} \int_{u_0}^{u}\|\Lb\Omega\varphi\|_{L^4(S_{\ub',u'})}\|L^2\varphi\|_{L^4(S_{\ub',u'})} \|L^2\Omega\varphi\|_{L^2(S_{\ub', u'})}d\ub' du'.
\end{align*}
We now use Sobolev inequality to control $L^2\varphi$:
\begin{align*}
\|L^2\varphi \|_{L^{4}(S_{\ub',u'})} &\lesssim|u'|^{\frac{1}{4}}\|L^2\varphi\|_{L^\infty}\\
&\lesssim  |u'|^{\frac{3}{4}} \|\nablaslash L^2\varphi\|_{L^{2}(S_{\ub', u'})} +|u'|^{-\frac{1}{4}}\| L^2\varphi \|_{L^{2}(S_{\ub', u'})}.
\end{align*}
Thus, we have
\begin{align*}
 S_{41}  &\lesssim \delta^{\frac{1}{4}} M \int_{u_0}^{u} \|L^2\Omega \varphi\|^2_{L^{2}(C_{u'})}  +\|L^2\varphi\|_{L^2(C_{u'})}\|L^2\Omega\varphi\|_{L^2(C_{u'})}  du'\\
&\lesssim \delta^{-\frac{7}{4}} \cdot M^3.
\end{align*}

The estimates for $S_{42}$ and $S_{43}$ are similar to $S_{41}$, we only give the conclusion:
\begin{equation*}
  S_{42}+S_{43} \lesssim \delta^{-\frac{3}{2}} \cdot M^3.
\end{equation*}

Therefore,
\begin{equation*}
  S_{4} \lesssim \delta^{-\frac{7}{4}} \cdot M^3.
\end{equation*}

For $S_{5}$, we have
\begin{align*}
 S_5 &\lesssim  \doubleint_{\mathcal{D}}|\Omega \varphi||\Lb L\varphi||L \varphi||L^2\Omega\varphi| + \doubleint_{\mathcal{D}}|\Omega \varphi||L^2\varphi||\Lb \varphi||L^2\Omega\varphi|\\
&\quad + \doubleint_{\mathcal{D}}|\Omega \varphi||\nablaslash L\varphi||\nablaslash \varphi||L^2\Omega\varphi|.
\end{align*}
Therefore, the estimates for $S_{5}$ is similar to $S_1$. The proof is routine and we only give the final estimates
\begin{align*}
 S_{5} \lesssim \delta^{-1} \cdot M^4.
\end{align*}

Similarly, the estimates for $S_{6}$ and $S_{7}$ are similar to those of $S_1$ or $S_2$, we only give the conclusions
\begin{align*}
 S_{6} +S_{7} \lesssim \delta^{-1} \cdot M^4.
\end{align*}

For $S_8$, $S_9$, $S_{10}$ and $S_{11}$, we simply bound each term in the quadratic expressions in $L^2$. We do this for $S_{10}$ to illustrate the idea:
\begin{align*}
 S_{10} &\lesssim |u|^{-1} \|\Lb L \Omega \varphi\|_{L^2(\mathcal{D})}\|L^2 \Omega \varphi\|_{L^2(\mathcal{D})}\\
&\lesssim \delta^{-1}  \cdot M^2.
\end{align*}
Similarly,
\begin{align*}
 S_{8} +S_{9} +S_{10} +S_{11}&\lesssim \delta^{-1} \cdot M^2.
\end{align*}

Putting all estimates for error terms back to \eqref{ES-F3-F4}, if $\delta$ is sufficiently small, we obtain
\begin{equation*}
\int_{C_{u}}|L^2 \Omega \varphi|^{2} \lesssim \delta^{-2}I^2  + \delta^{-\frac{7}{4}} M^3.
\end{equation*}
Hence, we obtain the estimates for $F_3(u,\ub)$:
\begin{equation}\label{ES-ES-F3-F4-b}
F_{3}(u,\ub) \lesssim I  + \delta^{\frac{1}{16}} M^{\frac{3}{2}}.
\end{equation}

\subsection{Estimates on $\Fb_3(u, \ub)$}\label{Section Estimates on Fb_3 and Fb_4}

In view of \eqref{commute L Omega n with main equation}, we ommute $\Lb$ and $\Omega$ with \eqref{Main Equation} to derive
\begin{equation*}
\begin{split}
 \Box \Lb \Omega \varphi = & \Lb\varphi \cdot Q_0(\nabla \Omega \varphi, \nabla \varphi)+ \Lb\Omega \varphi \cdot Q_0(\nabla \varphi, \nabla \varphi)\\
&+\varphi\cdot Q_0(\nabla \Lb\Omega \varphi, \nabla \varphi) + \varphi \cdot Q_0(\nabla \Omega \varphi, \nabla \Lb \varphi) + \Omega \varphi \cdot Q_0(\nabla \Lb\varphi, \nabla \varphi)\\
&+ \varphi\cdot \frac{1}{r}\cdot \left( 2Q_0(\nabla \Omega \varphi, \nabla \varphi) + L\Omega \varphi \cdot \Lb\varphi + L\varphi \cdot \Lb \Omega \varphi\right)\\
&+ \Omega \varphi\cdot \frac{2}{r}\cdot \left( Q_0(\nabla \varphi, \nabla \varphi) + L\varphi \cdot \Lb \varphi\right)+ \frac{1}{2r^2}(\Lb \Omega \varphi -L \Omega \varphi)-\frac{2}{r}\laplacianslash  \varphi.
\end{split}
\end{equation*}
Applying \eqref{fundamental energy identity} to the above equation with $\phi = \Lb\Omega\varphi,$ and $X = \Lb$, we obtain
\begin{equation}\label{ES-Fb3-Fb4-a}
\begin{split}
\int_{C_{u}}&|\nablaslash \Lb \Omega 	 \varphi|^{2}+\int_{\underline{C}_{\underline{u}}}|\Lb^2 \Omega\varphi|^{2} =\int_{C_{u_{0}}}|\nablaslash \Lb \Omega \varphi|^{2} + \doubleint_{\mathcal{D}}\Lb\varphi \cdot Q_0(\nabla \Omega \varphi, \nabla \varphi) \cdot \Lb^2\Omega\varphi\\
&+ \doubleint_{\mathcal{D}} \Lb\Omega \varphi \cdot Q_0(\nabla \varphi, \nabla \varphi) \cdot \Lb^2\Omega\varphi + \doubleint_{\mathcal{D}}\varphi\cdot Q_0(\nabla \Lb\Omega \varphi, \nabla \varphi)\cdot \Lb^2\Omega\varphi\\
&+ \doubleint_{\mathcal{D}}\varphi \cdot Q_0(\nabla \Omega \varphi, \nabla \Lb \varphi)  \cdot \Lb^2\Omega\varphi + \doubleint_{\mathcal{D}}\Omega \varphi \cdot Q_0(\nabla \Lb\varphi, \nabla \varphi)\cdot \Lb^2\Omega\varphi\\
&+ \doubleint_{\mathcal{D}}\varphi\cdot \frac{1}{r}\cdot \left( 2Q_0(\nabla \Omega \varphi, \nabla \varphi) + L\Omega \varphi \cdot \Lb\varphi + L\varphi \cdot \Lb \Omega \varphi\right) \cdot L^2\Omega\varphi\\
&+ \doubleint_{\mathcal{D}}\Omega \varphi\cdot \frac{2}{r}\cdot \left( Q_0(\nabla \varphi, \nabla \varphi) + L\varphi \cdot \Lb \varphi\right)\cdot \Lb^2\Omega\varphi\\
&+\doubleint_{\mathcal{D}}  \frac{1}{r^2}(\Lb \Omega \varphi -L \Omega \varphi)\cdot \Lb^2\Omega\varphi+\doubleint_{\mathcal{D}}\frac{2}{r}\laplacianslash  \varphi\cdot \Lb^2\Omega\varphi\\
&+\doubleint_{\mathcal{D}}\frac{1}{r}L \Lb \Omega \varphi \cdot \Lb^2\Omega\varphi+\doubleint_{\mathcal{D}}\frac{1}{r} |\nablaslash \Lb\Omega\varphi|^2.
\end{split}
\end{equation}
We rewrite the right hand side of the above equation as
\begin{equation*}
\int_{C_{u_{0}}}|\nablaslash \Lb \Omega \varphi|^{2} + S_1 +S_2 + \cdots+ S_{11},
\end{equation*}
where the error terms $S_i$'s are defined in the obvious way. We shall control the error terms one by one.

For $S_1$, we have
\begin{align*}
 S_1 &\lesssim \doubleint_{\mathcal{D}}|\Lb\varphi|^2 |L \Omega \varphi| |\Lb^2\Omega\varphi|+\doubleint_{\mathcal{D}}|L\varphi| |\Lb \varphi|| \Lb \Omega \varphi| |\Lb^2\Omega\varphi| + \doubleint_{\mathcal{D}}|\Lb\varphi| |\nablaslash \varphi| |\nablaslash \Omega \varphi| |\Lb^2\Omega\varphi|\\
&=S_{11} +S_{12} +S_{13}.
\end{align*}

The estimates on $S_{11}$, $S_{12}$ and $S_{13}$ are similar. We only work out the details for $S_{12}$:
\begin{align*}
 S_{12}  &\lesssim \|\Lb\varphi\|_{L^\infty}\|L\varphi\|_{L^\infty} \int_{0}^{\ub} \|\Lb\Omega \varphi\|_{L^2(\Cb_{\ub'})} \|\Lb^2\Omega\varphi\|_{L^2(\Cb_{\ub'})} d \ub'\\
&\lesssim \delta^{\frac{3}{4}} \cdot M^4.
\end{align*}
This eventually leads to
\begin{align*}
 S_{1} \lesssim \delta^{\frac{3}{4}} \cdot M^4.
\end{align*}

For $S_2$, we have
\begin{align*}
 S_2 &\lesssim \doubleint_{\mathcal{D}}|L\varphi||\Lb \varphi| |\Lb \Omega \varphi| |\Lb^2\Omega\varphi|+\doubleint_{\mathcal{D}}|\nablaslash\varphi|^2 |\Lb \Omega \varphi| |\Lb^2\Omega\varphi|.
\end{align*}
It can be estimated in the same way as $S_{11}$, therefore,
\begin{align*}
 S_{2} \lesssim \delta^{\frac{5}{4}} \cdot M^4.
\end{align*}

For $S_3$, we have
\begin{align*}
 S_3 &\lesssim  \doubleint_{\mathcal{D}}|L \varphi| |\Lb^2\Omega\varphi|^2 + \doubleint_{\mathcal{D}}|\Lb\varphi||L\Lb\Omega\varphi| |\Lb^2\Omega\varphi| + \doubleint_{\mathcal{D}}|\nablaslash\varphi||\nablaslash \Lb\Omega\varphi| |\Lb^2\Omega\varphi|.
\end{align*}
It can be estimated in the same way as $S_{11}$, therefore,
\begin{align*}
 S_{3} \lesssim \delta^{\frac{1}{2}} \cdot M^3.
\end{align*}

For $S_{4}$, we have
\begin{align*}
 S_{4} &\lesssim  \doubleint_{\mathcal{D}}|L \Omega\varphi||\Lb^2 \varphi||\Lb^2\Omega\varphi| + \doubleint_{\mathcal{D}}|\Lb\Omega\varphi| |L \Lb \varphi||\Lb^2\Omega\varphi| + \doubleint_{\mathcal{D}}|\nablaslash \Omega \varphi||\nablaslash \Lb \varphi| |\Lb^2\Omega\varphi|\\
&=S_{41} +S_{42} +S_{43}.
\end{align*}

For $S_{41}$, we have
\begin{align*}
 S_{41}  &\leq \int_0^{\ub} \int_{u_0}^{u}\|L \Omega\varphi\|_{L^4(S_{\ub',u'})}\|\Lb^2\varphi\|_{L^4(S_{\ub',u'})} \|\Lb^2\Omega\varphi\|_{L^2(S_{\ub', u'})}d\ub' du'.
\end{align*}
According to Sobolev inequality, we have
\begin{align*}
\|\Lb^2\varphi \|_{L^{4}(S_{\ub',u'})} &\lesssim|u'|^{\frac{1}{4}}\|\Lb^2\varphi\|_{L^\infty}\\
&\lesssim  |u'|^{\frac{3}{4}} \|\nablaslash \Lb^2\varphi\|_{L^{2}(S_{\ub', u'})} +|u'|^{-\frac{1}{4}}\| \Lb^2\varphi \|_{L^{2}(S_{\ub', u'})}.
\end{align*}
Thus, we have
\begin{align*}
 S_{41}  &\lesssim \delta^{\frac{1}{2}} \cdot M^3.
\end{align*}
The estimates for $S_{42}$ and $S_{43}$ are similar to $S_{41}$, eventually, we have
\begin{equation*}
  S_{4} \lesssim  \delta^{\frac{1}{2}} \cdot M^3.
\end{equation*}

For $S_{5}$, $S_{6}$ and $S_{7}$, the estimates are similar to those of $S_1$ or $S_2$, we only give the conclusions
\begin{align*}
 S_5+ S_{6} +S_{7} \lesssim \delta^{\frac{1}{2}} \cdot M^4.
\end{align*}

For $S_8$, $S_9$, $S_{10}$ and $S_{11}$, once again, as in last subsection, we simply bound each term in the quadratic expressions in $L^2$. This yields
\begin{align*}
 S_{8} +S_{9} +S_{10} +S_{11}&\lesssim \delta^{\frac{1}{2}}  \cdot M^2.
\end{align*}

Putting all estimates together, if $\delta$ is sufficiently small, we obtain
\begin{equation*}
\int_{\underline{C}_{\underline{u}}}|\Lb^2 \Omega\varphi|^{2} \lesssim \delta^{2} I^2  + \delta^{\frac{1}{2}} \cdot M^3.
\end{equation*}
Finally, we have
\begin{equation}\label{ES-Fb3-Fb4-b}
 \Fb_{3}(u, \ub) \lesssim \delta I +\delta^{\frac{1}{8}} \cdot M^{\frac{3}{2}}.
\end{equation}

\subsection{End of the Bootstrap Argument}\label{Section End of the Bootstrap Argument}
We collect \eqref{ES E-Eb up to three derivative}, \eqref{ES two derivative F2 Fb2}, \eqref{ES-ES-F3-F4-b} and \eqref{ES-Fb3-Fb4-b} as follows:
\begin{equation}\label{End of the Bootstrap Argument}
M \lesssim I + \delta^{\frac{1}{16}} \cdot M^2
\end{equation}
Choosing $\delta$ sufficiently small depending on the quantities $I$, we conclude that
\begin{equation*}
 M \lesssim I.
\end{equation*}
This completes the proof of \textbf{Main A priori Estimates} in this section.

\subsection{Higher Order Derivative Estimates}\label{Section Higher Order Derivative Estimates}

For higher order derivative estimates, the argument is completely analogous, even simper, because we have already closed the bootstrap argument and we can simply use an induction argument to derive estimates for each order. Therefore, we shall omit the detail and only sketch the proof. We introduce a family of energy norms similarly as before:
\begin{equation*}
\begin{split}
E_{k}(u,\ub) &= \|L \nablaslash^{k-1} \varphi\|_{L^2(C_u)} + \delta^{-\frac{1}{2}}  \|\nablaslash^{k} \varphi\|_{L^2(C_u)},\\
\Eb_{k}(u,\ub) &= \|\nablaslash^{k} \varphi\|_{L^2(\Cb_{\ub})} + \delta^{-\frac{1}{2}} \| \Lb \nablaslash^{k-1} \varphi\|_{L^2(\Cb_{\ub})}.
\end{split}
\end{equation*}
We also need another family of norms which involves at least two null derivatives. They are defined as follows,
\begin{equation*}
F_{k}(u,\ub) =\delta \|L^2 \nablaslash^{k-2}\varphi\|_{L^2(C_u)}, \quad \Fb_{k}(u,\ub) = \| \Lb^2 \nablaslash^{k-2}\varphi\|_{L^2(\Cb_{\ub})}.
\end{equation*}

We derive the estimates by induction on the number of derivatives $k$. For $k=1, 2, 3$,  they corresponding quantities are all bounded by a universal constant depending only on the initial data:
\begin{equation*}
\sum_{i=1}^3 [E_i(u,\ub) + \Eb_1(u,\ub)]+\sum_{j=2}^3 [F_j(u,\ub) + \Fb_j(u,\ub)] \leq C(I),
\end{equation*}

We want to show that for all initial data of \eqref{Main Equation} and all $I \in \mathbb{R}_{>0}$ which satisfy
\begin{equation*}
\sum_{i=1}^{n+2} E_i(u_{0},\delta)  + \sum_{j=2}^{n+2} F_j(u_{0},\delta)  \leq I,
\end{equation*}
there is a constant $C(I,n)$ depending only on $I$ and $n$, if $\delta$ is sufficiently small (depending on $n$), so that
\begin{equation*}
 [E_{n+2}(u,\ub) + \Eb_{n+2}(u,\ub)]+ [F_{n+2}(u,\ub) + \Fb_{n+2}(u,\ub)] \leq C(I,n),
\end{equation*}
for all $u \in [u_0, u^*]$ and $\ub \in [0, \ub^*]$.

At this stage, we can use the induction assumption that
\begin{equation*}
\sum_{i=1}^{n+1} [E_i(u,\ub) + \Eb_i(u,\ub)]+\sum_{j=2}^{n+1} [F_j(u,\ub) + \Fb_j(u,\ub)] \leq C(I,n).
\end{equation*}
and make the following bootstrap assumption:
\begin{equation}
[E_{n+2}(u,\ub) + \Eb_{n+2}(u,\ub)]+ [F_{n+2}(u,\ub) + \Fb_{n+2}(u,\ub)] \leq M,
\end{equation}
for all $u \in [u_0, u^*]$ and $\ub \in [0, \ub^*]$, where $M$ is a sufficiently large constant.

We then proceed exactly as before: we first derive  $L^\infty$ for derivatives of order less or equal to $n$ as well as $L^4$ estimates for derivatives of order less or equal to $n+1$. In fact, for $2\leq k\leq n$, we have
\begin{equation*}
 \delta^{\frac{1}{2}} \|L\nablaslash^{k-1}\varphi\|_{L^\infty} + \delta^{-\frac{1}{4}} \|\nablaslash^{k}\varphi\|_{L^\infty} +  \delta^{-\frac{1}{4}} \| \Lb\nablaslash^{k-1}\varphi\|_{L^\infty} \lesssim M,
\end{equation*}
and
\begin{equation*}
\delta^{\frac{1}{2}} \| L\nablaslash^{k} \varphi \|_{L^{4}(S_{\ub,u})} +\delta^{-\frac{1}{4}} \|\nablaslash^{k+1} \varphi\|_{L^{4}(S_{\ub,u})} + \delta^{-\frac{1}{4}} \| \Lb\nablaslash^{k}\varphi \|_{L^{4}(S_{\ub,u})} \lesssim M.
\end{equation*}
We then can derive energy estimates use a similar argument as before to derive
\begin{equation*}
\begin{split}
E_{n+2}(u,\ub) + \Eb_{n+2}(u,\ub) &\lesssim I  +  \delta^{\frac{1}{16}}  M^{2},\\
F_{n+2}(u,\ub) & \lesssim I  + \delta^{\frac{1}{16}}M^{\frac{3}{2}},\\
\Fb_{n+2}(u,\ub) &\lesssim  \delta \cdot I+\delta^{\frac{1}{8}} M^{\frac{3}{2}}.
\end{split}
\end{equation*}
Finally, we can take a sufficiently small $\delta$ to obtain the estimates.

\section{Existence of Solutions}
In this section, based on the a priori estimates in last section, we first show that \eqref{Main Equation} with data prescribed on $C_{u_0}$ where $u_0 \leq \ub \leq \delta$ in last sections can be solved all the way up the $t=-1$, i.e. the Region $2$. Recall that Region $2$ is in the future domain of dependence of $\Cb_0$ and $C_{u_0}$ (with $0 \leq \ub \leq \delta$) and the data on $\Cb_0$ is completely trivial.

To start, we use the local existence result \cite{R-90} of Rendall for semi-linear wave equations for characteristic data, we know that there exists a solution around $S_{0, u_0}$, say, defined in the region enclosed by $\Cb_{0}$, $C_{u_0}$ and $t = u_0 + \varepsilon$ with $\varepsilon << \delta$. Thanks to the a priori estimates, if at the beginning we assume the bound on data for at least $8$ derivatives, the $L^\infty$ norms of at least up to $6$ derivatives of the solution are bounded by the data on $t = u_0 + \varepsilon$. Therefore, we can solve a Cauchy problem with data prescribed on $t = u_0 + \varepsilon$ to construct a solution in the future domain dependence of $t = u_0 + \varepsilon$ whose boundary have two null hypersurface $C_{u_0 + \varepsilon}$ and $\Cb_{\varepsilon}$. Now we have two characteristic problem: for the first one, the data is prescribed on $\Cb_0$ and $C_{u_0 + \varepsilon}$; for the second one, the data is prescribed on $C_{u_0}$ and $\Cb_{\varepsilon}$. We can use Rendall's local existence result again to solve them around $S_{0, u_0+\varepsilon}$ and $S_{\varepsilon, u_0}$. In this way, we can actually push the solution to $t = u_0 + \varepsilon + \varepsilon'$ with another small $\varepsilon'$.

We then can repeat the above process in an obvious way to push the solution all the way to $t = u_0 + \delta$. Similarly, we can then push it from $t = u_0 + \delta$ to $t= -1$. Therefore, we have constructed a solution in the entire Region $2$. We remark that this process depends crucially on the a priori estimates since the $L^\infty$ norms of the derivatives of $\varphi$ is guaranteed to be bounded.

We now turn to the study of the solution $\varphi$ on $\Cb_{\delta}$ and it will be important for the construction of solution in Region $3$.
\begin{proposition}\label{data on Cb_delta} Assume we have bound on $E_i(u_0)$ with $i=1,2,\cdots, n+2$ and $F_j(u_0)$ with $j=2,3, \cdots, n+2$ for some fixed $n \geq 8$. Then, for $k=1,2, \cdots, n$, we have
\begin{equation*}
\begin{split}
\|\nablaslash^k \varphi\|_{L^\infty} + \|\Lb \nablaslash^{k-1} \varphi\|_{L^\infty}+\|\Lb^2 \nablaslash^{k} \varphi\|_{L^\infty} \lesssim \delta^{\frac{1}{4}},\\
\|L \nablaslash^{k-1} \varphi\|_{L^\infty}+\|L^2\nablaslash^{k} \varphi\|_{L^\infty} \lesssim \delta^{\frac{1}{4}}.
\end{split}
\end{equation*}
\end{proposition}
\begin{proof}
The first inequality is standard. Its proof is basically integration along $L$ directions and using the higher order derivative estimates derived in last section.

The second one is a little bit surprising, since we expect $L$ derivative cause a loss of $\delta^{-\frac{1}{2}}$. The idea is that the loss in $\delta$ only should occur from initial data but not from the energy estimates. Recall that the data is given by
\begin{equation*}
 \varphi(\ub, u_0, \theta) = \frac{\delta^{\frac{1}{2}}\psi_0 (\frac{\ub}{\delta}, \theta) + (0,0,1)}{\sqrt{\delta|\psi_0 (\frac{\ub}{\delta}, \theta)|^2 +1}},
\end{equation*}
where the $\ub$-support of $\psi_0$ is inside $(0,1)$. Therefore, on $C_{u_0}$ near $S_{0,u_0}$, the data is completely trivial. In particular, $(L^i \nablaslash^j \varphi)(u_0, \delta, \theta) \equiv 0$. We then integrate \eqref{Main Equation} to get estimates on $\Cb_{\delta}$.

To illustrate the above idea, we now prove
\begin{equation*}
\|L \varphi\|_{L^\infty} \lesssim \delta^{\frac{1}{4}}.
\end{equation*}
The rest of the inequalities can be derived exactly in the same way. We will not give the proof.

We rewrite \eqref{Main Equation} as
\begin{equation*}
 \Lb L \varphi  - \frac{1}{2r}L\varphi-\varphi \Lb\varphi \cdot L\varphi =\laplacianslash \varphi-\frac{1}{2r}\Lb\varphi
-\varphi |\nablaslash \varphi|^2.
\end{equation*}
This can be view as an ODE for $L\varphi$ on $\Cb_\delta$ with trivial data on $S_{\delta,u_0}$. Since $\nablaslash^i \varphi$ and $\Lb \varphi$ are all of size $\delta^{\frac{1}{4}}$, we can integrate above equation and use Gronwall's inequality to derive
\begin{equation*}
\|L \varphi\|_{L^\infty} \lesssim \delta^{\frac{1}{4}}.
\end{equation*}
This completes the proof.
\end{proof}
\begin{remark}
According to the above proposition, the data on $\Cb_\delta$ induced from Region $2$ are small in energy norms. Note also that the above process to gain smallness actually loses one derivative through the integration.
\end{remark}

We now choose a Cauchy hypersurface $\Sigma = \{t = u_0 + \delta\}$. Let $\Sigma_1 = \Sigma \cap (\text{Region}\,\,1 \cup \text{Region}\,\,2)$ and $\Sigma_2 = \Sigma - \Sigma_1$. We can restrict the solution constructed in previous sections on $\Sigma_1$ to get
\begin{equation*}
(\varphi,\partial_t \varphi)|_{\Sigma_1} = (\varphi_1^{(0)}, \varphi_1^{(1)}).
\end{equation*}

According to the estimates derived in the above Propostion, we have the following properties for $(\varphi_1^{(0)}, \varphi_1^{(1)})$:
\begin{equation*}
\begin{split}
(\varphi_1^{(0)}, \varphi_1^{(1)})|_{\partial \Sigma_1} &= \text{constants};\\
\|(\partial^k \varphi_1^{(0)}, \partial^{k-1} \varphi_1^{(1)})\|_{L^\infty({\partial \Sigma_1})} &\lesssim \delta^{\frac{1}{4}}, \quad \text{for} \,\, k=1,2,\cdots,8.
\end{split}
\end{equation*}
We then can apply Whitney extension theorem (see Theorem $12$ in \cite{Luli} and references therein) to extend $(\varphi_1^{(0)}, \varphi_1^{(1)})$ to the entire $\Sigma$ to obtain Cauchy data $(\varphi^{(0)}, \varphi^{(1)})$ with the following properties:
\begin{equation*}
\begin{split}
(\varphi^{(0)}, \varphi^{(1)})|_{\Sigma_1} &=(\varphi_1^{(0)}, \varphi_1^{(1)});\\
(\varphi^{(0)}, \varphi^{(1)})|_{\{x \in \Sigma_2| \text{dist}(x,\Sigma_1)\geq 1\}} &=((0,0,1), (0,0,0));\\
\|(\partial^k \varphi^{(0)}, \partial^{k-1}\varphi^{(1)})\|_{L^\infty({\{x \in \Sigma_2| \text{dist}(x,\Sigma_1)\leq 1\}})} &\lesssim \delta^{\frac{1}{4}}, \quad \text{for} \,\, k=1,2,\cdots,7.
\end{split}
\end{equation*}
In particular, the (conformal) energies of $(\varphi^{(0)}, \varphi^{(1)})|_{\Sigma_2}$ up to seven derivatives are small (of order $\delta$). Therefore, according to the small data theory, we obtain a solution $\varphi$ in Region $4$. In particular, the energy flux on $C_{u_0}^+$ induced form the solution in Region $4$ are small.

We now have the data on $\Cb_\delta$ and $C_{u_0}^+$. They are past boundaries of Region $3$. We can then solve this small data problem on Region $3$. Together with the solutions constructed in other regions, this completes the proof of the main theorem.

\begin{remark}
In fact, to construct such initial data set, we do not have to take data in such a form as \eqref{precise short pulse data on C_0}. According to Section \ref{Section A priori Estimates}, if the data satisfy \eqref{L2 on initial surface 1} and \eqref{L2 on initial surface 2}, then the a priori estimates hold. Therefore, we can still construct solutions. In this way, we can construct an open set of data for which the main theorem of the paper is still valid.
\end{remark}

\end{document}